\documentclass[11pt,a4paper]{article}

\usepackage{amssymb}
\usepackage{amsfonts}
\usepackage{amsmath, amsthm}
\usepackage{fullpage}
\usepackage[all]{xy}
\usepackage{multirow}
\usepackage{pgf}

\newtheorem{theorem}{Theorem}[section]
\newtheorem{lemma}[theorem]{Lemma}
\newtheorem{proposition}[theorem]{Proposition}
\newtheorem{corollary}[theorem]{Corollary}
\newtheorem{remark}[theorem]{Remark}

\newtheorem{introTheorem}{Theorem}
\newtheorem{introCorollary}[introTheorem]{Corollary}

\newcommand{\N}{\mathbb{N}}
\newcommand{\Z}{\mathbb{Z}}
\newcommand{\Q}{\mathbb{Q}}
\newcommand{\R}{\mathbb{R}}
\newcommand{\C}{\mathbb{C}}

\newcommand{\typeA}{\mathsf{A}}
\newcommand{\typeB}{\mathsf{B}}
\newcommand{\typeC}{\mathsf{C}}
\newcommand{\typeD}{\mathsf{D}}
\newcommand{\typeE}{\mathsf{E}}
\newcommand{\typeF}{\mathsf{F}}
\newcommand{\typeG}{\mathsf{G}}

\newcommand{\coomega}{\check{\omega}}
\newcommand{\coalpha}{\check{\alpha}}
\newcommand{\coepsilon}{\check{\epsilon}}
\newcommand{\RP}{\mathcal{P}}

\newcommand{\affPhi}{\widetilde{\Phi}}
\newcommand{\affDelta}{\widetilde{\Delta}}
\newcommand{\mi}{\mathop{\mathsf{min}}}
\newcommand{\prmi}{\mathop{\mathsf{pmin}}}
\newcommand{\IRR}{\mathcal{I}}

\newcommand{\lcm}{\mathop{\mathrm{lcm}}}
\newcommand{\polyhedral}{polyhedral}

\newcommand{\radius}{3.5pt}

\newcommand{\hiddenColor}{0.8,0.8,0.8}
\newcommand{\positiveColor}{0.9,0.9,0.9}

\makeatletter
\def\blfootnote{\xdef\@thefnmark{}\@footnotetext}
\makeatother

\title{Root polytope and partitions}

\author{Rocco Chiriv\`\i}

\begin{document}

\maketitle

\abstract{Given a crystallographic reduced root system and an element $\gamma$ of the lattice generated by the roots, we study the minimum number $|\gamma|$, called the length of $\gamma$, of roots needed to express $\gamma$ as sum of roots. This number is related to the linear functionals presenting the convex hull of the roots. The map $\gamma\longmapsto|\gamma|$ turns out to be the upper integral part of a piecewise linear function with linearity domains the cones over the facets of this convex hull. In order to show this relation we investigate the integral closure of the monoid generated by the roots in a facet. We study also the positive length, i.e. the minimum number of positive roots needed to write an element, and we prove that the two notions of length coincide only for the types $\typeA_\ell$ and $\typeC_\ell$.}

\setcounter{section}{0}

\section{Introduction}

\blfootnote{
{\it 2010 Mathematics Subject Classiﬁcation.} Primary 05E45; Secondary 17B22.

$\quad${\it Key words and phrases.} Root system, root polytope, partitions.
}

\renewcommand{\theintroTheorem}{\Alph{introTheorem}}

Let $\Phi$ be a crystallographic reduced root system in the euclidean space $E$ with scalar product $(\cdot,\cdot)$. Let $R$ be the $\Z$--span of $\Phi$ in $E$ and define $|\gamma|$, the \emph{length} of $\gamma$, as the minimum $r\geq0$ such that there exist $r$ roots $\beta_1,\,\beta_2,\ldots,\beta_r$ with $\gamma=\beta_1+\beta_2+\cdots+\beta_r$. So $|\gamma|$ is the size of a minimal partition of $\gamma$ in roots.

Our aim is to describe the map $R\ni\gamma\longmapsto|\gamma|\in\N$. (It is called the word length with respect to $\Phi$ in \cite{ardila}, see also \cite{sturmfels}, and \cite{wangZhao}). As we show it is related to the convex hull $\RP_\Phi$ of $\Phi$ in $E$, called the \emph{root polytope}. Given an element $\lambda\in E$, let $H(\lambda)$ be the closed half--space of $E$ defined by $\{u\in E\,|\,(\lambda,u)\leq1\}$. If $F$ is a facet of $\RP_\Phi$ let $\lambda_F\in E$ be such that $(\lambda_F,u)=1$ if $u\in F$ and $(\lambda_F,u)<1$ for $u\in\RP_\Phi\setminus F$; then $\RP_\Phi=\bigcap_F H(\lambda_F)$, where $F$ ranges over the set of all facets of $\RP_\Phi$, is a half--space presentation of $\RP_\Phi$. Moreover let $V(F)\doteq F\cap\Phi$ be the set of roots in $F$ and $C(V(F))$ be the $\Q^+$--cone over $V(F)$, i.e. the set of non--negative rational linear combinations of elements of $V(F)$. Our main result is the following formula.

\begin{introTheorem}\label{introTheorem_formula} For any $\gamma\in R$ we have
$$
|\gamma|=\max_F\lceil(\lambda_F,\gamma)\rceil.
$$
So if $\gamma\in C(V(F))$ for some facet $F$ then $|\gamma|=\lceil(\lambda_F,\gamma)\rceil$.
\end{introTheorem}

Hence we see that the length map is the upper integral part of linear functions on the cones over the facets of $\RP_\Phi$.

Our proof of this theorem requires the evaluation of the map $\gamma\longmapsto|\gamma|$ at {a minimal set of} generators of the monoid $M(F)\doteq C(V(F))\cap R$, with $F$ a face of $\RP_\Phi$. Let $N(F)$ be the $\N$--span of $V(F)$ and $Z(F)$ be the $\Z$--span of $V(F)$; $N(F)$ is a submonoid of $M(F)$ and we say that the face $F$ is \emph{normal} if $C(V(F))\cap Z(F)=N(F)$. Notice, however, that it is to be expected that $M(F)$ is larger than $N(F)$ since for some facets $Z(F)$ is a proper sublattice of $R$. We call $M(F)$ the \emph{integral closure} of $N(F)$ in $R$ and we say that $F$ is \emph{integrally closed} in $R$ if $M(F)=N(F)$. The relation of these monoids to certain toric varieties gives reason to these definitions; in particular the normality of a face $F$ is equivalent to the normality of the toric variety whose coordinate ring is $\C[t^\beta\,|\,\beta\in V(F)]$. Similar varieties have been extensively studied over the years; see for example \cite{procesi} and the other papers cited there.

As stated above, we need to find the generators of $M(F)$; this computation uses the uniform description of $\RP_\Phi$ given by Cellini and Marietti in \cite{celliniMarietti} (see also Vinberg's paper \cite{vinberg} and \cite{chari}, \cite{khare} for a generalization). Assuming here that $\Phi$ is irreducible, the faces of $\RP_\Phi$ may be naturally defined in terms of the affine root system associated to $\Phi$. Let us say that a simple root $\alpha$ is \emph{maximal} if its complement in the affine Dynkin diagram of $\Phi$ is connected (see Table \ref{table_rootData} in Section \ref{section_normality} for the list of maximal roots). Let $\coomega_\alpha$ be the coweight dual to $\alpha$ and let $\theta$ be the highest root of $\Phi$; then the \emph{standard parabolic facet} $F(\alpha)$, with $\alpha$ a (simple) maximal root, is the set of elements $u\in E$ such that $(\coomega_\alpha,u)=(\coomega_\alpha,\theta)$. As we recall in Section \ref{section_rootPolytope} below, any facet of $\RP_\Phi$ is in the orbit of a unique standard parabolic facet by the Weyl group action. Notice that this gives in particular an explicit half--space presentation of $\RP_\Phi$ and allows for an effective computation of the length as in the above theorem.

We are now ready to report our computation about the generators of $M(F)$. As we see in Section \ref{section_preliminary} the intersection of $M(F)$ with a face of the cone $C(V(F))$ is the monoid $M(F')$ for some facet $F'$ of a subsystem of $\Phi$; hence we consider only the \emph{proper} generators of $M(F)$, i.e. those not in the border of $C(V(F))$. In the following theorem we number the simple roots and the fundamental weights as in \cite{bourbaki}.

\begin{introTheorem}\label{introTheorem_properMinimalElements} Let $F$ be a facet of $\RP_\Phi$, then the proper generators for the monoid $M(F)$ not in $V(F)$ are as follows:
\begin{itemize}
\item $2\omega_3$ for the facet $F(\alpha_3)$ of type $\typeB_3$,
\item $2\omega_2$ for the facet $F(\alpha_2)$ of type $\typeE_7$,
\item $\omega_2$ and $2\omega_2$ for the facet $F(\alpha_2)$ of type $\typeE_8$ and
\item $\omega_1$ and $2\omega_1$ for the facet $F(\alpha_1)$ of type $\typeG_2$.
\end{itemize}
All other facets of any other type have no proper generator not in $V(F)$.
\end{introTheorem}

In Section \ref{section_normality} we develop a general theory for proper generators to a certain extent; but the proof of the above theorem needs some simple computations and checks that are carried out on a case--by--case basis. As a consequence of the theorem we have the following two results.

\begin{introCorollary}\label{introCorollary_faceNormality} Any face of the root polytope $\RP_\Phi$ is normal.
\end{introCorollary}

\begin{introCorollary}\label{introCorollary_faceIntegrality} The facet $F(\alpha)$, $\alpha\in\Delta$ a maximal root, is integrally closed in $R$ if and only if $(\coomega_\alpha,\theta)=1$.
\end{introCorollary}

These integral closure properties may also be proved by finding a unimodular triangulation of the facets of the root polytope; see for example \cite{ardila} where such a triangulation is given for type $\typeA$, $\typeC$ and $\typeD$ via explicit realizations of these root systems. Similar polytopes related to root systems and their normality are studied via unimodular triangulation in \cite{bruns}, \cite{lamI} and \cite{lamII} while in \cite{payne} a combinatorial characterization of diagonally split toric varieties is used. Notice however, that, to our best knowledge, it is not known whether all root polytopes admit an unimodular triangulation suitable to prove the above corollaries, nor we know how to costruct a triangulation in a uniform way with respect to the root system type.

It is natural to consider also another type of length map. Let us choose a positive subsystem $\Phi^+$ of $\Phi$ and let $R^+$ be the $\N$--span of $\Phi^+$ and, to an element $\gamma$ of $R^+$, let us associate the minimum number $|\gamma|_+$ of positive roots needed to write $\gamma$ as sum of positive roots. So $|\gamma|_+$ is the size of a minimal partition of $\gamma$ in positive roots. We call $|\gamma|_+$ the \emph{positive length} of $\gamma$.  It is clear that $|\gamma|\leq|\gamma|_+$ but in general this inequality is strict. Consider, for example, $\typeB_3$ and the root $\beta\doteq\alpha_1+\alpha_2+2\alpha_3$. We have $\gamma\doteq\beta-\alpha_2=\alpha_1+2\alpha_3$ and this shows that $|\gamma|=2$ while $|\gamma|_+=3$ since any root has connected support.

\begin{introCorollary}\label{introCorollary_lengths}
The positive length map coincides with the length map only for the types $\typeA_\ell$ and $\typeC_\ell$.
\end{introCorollary}

For type $\typeA_\ell$ we prove this theorem by comparing a direct formula for the positive length (Proposition \ref{proposition_lengthMapA}) with the formula in Theorem \ref{introTheorem_formula}. For $\typeC_\ell$ we use a different strategy exploiting a triangulation of the root polytope described in \cite{celliniMarietti2}.

The different behavior of types $\typeA_\ell$ and $\typeC_\ell$ with respect to the length map is reflected in the following compatibility condition. Let us denote by $\RP_\Phi^+$ the convex hull of the set $\Phi^+\cup\{0\}$, called the \emph{positive root polytope}, and let also $C(\Phi^+)$ be the non--negative rational cone generated by the positive roots. We say that $\Phi^+$ is \emph{\polyhedral} if $C(\Phi^+)$ is a union of cones generated by subsets of roots of the faces of $\RP_\Phi$. Equivalently, $\Phi^+$ is \polyhedral\ if $\RP_\Phi^+=\RP_\Phi\cap C(\Phi^+)$.

\begin{figure}
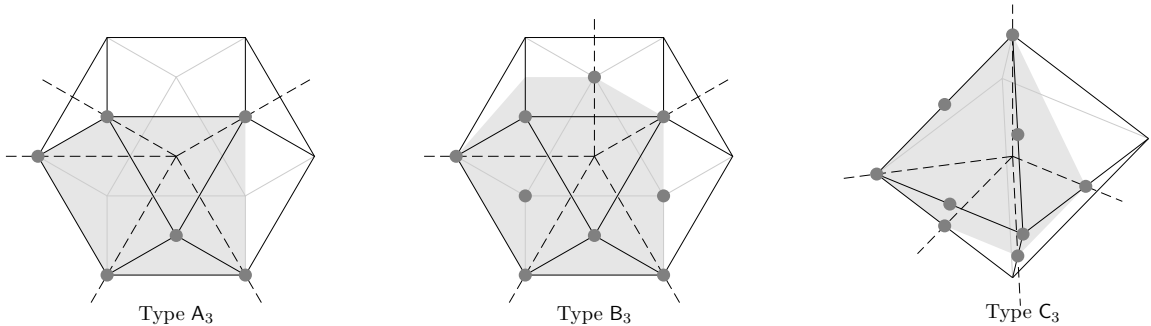

% rettangolo contenitore
\begin{pgfpicture}{0cm}{0cm}{16.5cm}{5cm}
% A3
% trasla
\begin{pgftranslate}{\pgfpoint{2.75cm}{2.5cm}}
% scala
\begin{pgfmagnify}{0.7}{0.7}
% setta gi assi
\pgfsetxvec{\pgfpoint{1.299cm}{-0.75cm}}
\pgfsetyvec{\pgfpoint{0cm}{1.5cm}}
\pgfsetzvec{\pgfpoint{-1.299cm}{-0.75cm}}
% colora la parte positiva del politopo
\begin{pgfscope}
\color[rgb]{\positiveColor}
\pgfmoveto{\pgfxyz(0,1,1)}
\pgflineto{\pgfxyz(1,1,0)}
\pgflineto{\pgfxyz(1,-1,0)}
\pgflineto{\pgfxyz(0,-1,1)}
\pgflineto{\pgfxyz(-1,0,1)}
\pgflineto{\pgfxyz(0,1,1)}
\pgffill
\end{pgfscope}
% lati nascosti
\begin{pgfscope}
\color[rgb]{\hiddenColor}
\pgfmoveto{\pgfxyz(-1,1,0)}
\pgflineto{\pgfxyz(-1,0,-1)}
\pgflineto{\pgfxyz(0,-1,-1)}
\pgflineto{\pgfxyz(-1,-1,0)}
\pgflineto{\pgfxyz(-1,0,-1)}
\pgflineto{\pgfxyz(0,1,-1)}
\pgfstroke
\pgfmoveto{\pgfxyz(-1,0,1)}
\pgflineto{\pgfxyz(-1,-1,0)}
\pgflineto{\pgfxyz(0,-1,1)}
\pgfstroke
\pgfmoveto{\pgfxyz(1,0,-1)}
\pgflineto{\pgfxyz(0,-1,-1)}
\pgflineto{\pgfxyz(1,-1,0)}
\pgfstroke
\end{pgfscope}
% spigoli del cono positivo
\begin{pgfscope}
\pgfsetdash{{0.2cm}{0.1cm}}{0cm}
\pgfline{\pgfxyz(0,0,0)}{\pgfxyz(2,2,0)}
\pgfline{\pgfxyz(0,0,0)}{\pgfxyz(0,2,2)}
\pgfline{\pgfxyz(0,0,0)}{\pgfxyz(-1.25,0,1.25)}
\pgfline{\pgfxyz(0,0,0)}{\pgfxyz(0,-1.25,1.25)}
\pgfline{\pgfxyz(0,0,0)}{\pgfxyz(1.25,-1.25,0)}
\end{pgfscope}
% lati visibili
\pgfmoveto{\pgfxyz(1,1,0)}
\pgflineto{\pgfxyz(0,1,-1)}
\pgflineto{\pgfxyz(-1,1,0)}
\pgflineto{\pgfxyz(0,1,1)}
\pgflineto{\pgfxyz(1,1,0)}
\pgflineto{\pgfxyz(1,0,-1)}
\pgflineto{\pgfxyz(1,-1,0)}
\pgflineto{\pgfxyz(0,-1,1)}
\pgflineto{\pgfxyz(1,0,1)}
\pgflineto{\pgfxyz(0,1,1)}
\pgflineto{\pgfxyz(-1,0,1)}
\pgflineto{\pgfxyz(0,-1,1)}
\pgfstroke
\pgfmoveto{\pgfxyz(1,1,0)}
\pgflineto{\pgfxyz(1,0,1)}
\pgflineto{\pgfxyz(1,-1,0)}
\pgfstroke
\pgfline{\pgfxyz(1,0,-1)}{\pgfxyz(0,1,-1)}
\pgfline{\pgfxyz(-1,1,0)}{\pgfxyz(-1,0,1)}
% radici positive
\begin{pgfscope}
\color[rgb]{0.5,0.5,0.5}
\pgfcircle[fill]{\pgfxyz(1,1,0)}{\radius}
\pgfcircle[fill]{\pgfxyz(0,1,1)}{\radius}
\pgfcircle[fill]{\pgfxyz(1,0,1)}{\radius}
\pgfcircle[fill]{\pgfxyz(1,-1,0)}{\radius}
\pgfcircle[fill]{\pgfxyz(0,-1,1)}{\radius}
\pgfcircle[fill]{\pgfxyz(-1,0,1)}{\radius}
\end{pgfscope}
% testo
\pgfputat{\pgfxyz(1,-1.0,1)}{\pgfbox[center,center]{Type $\typeA_3$}}
\end{pgfmagnify}
\end{pgftranslate}
% B3
% trasla
\begin{pgftranslate}{\pgfpoint{8.25cm}{2.5cm}}
% scala
\begin{pgfmagnify}{0.7}{0.7}
% setta gi assi
\pgfsetxvec{\pgfpoint{1.299cm}{-0.75cm}}
\pgfsetyvec{\pgfpoint{0cm}{1.5cm}}
\pgfsetzvec{\pgfpoint{-1.299cm}{-0.75cm}}
% colora la parte positiva del politopo
\begin{pgfscope}
\color{lightgray}
\color[rgb]{\positiveColor}
\pgfmoveto{\pgfxyz(-0.5,1,0.5)}
\pgflineto{\pgfxyz(0,1,0)}
\pgflineto{\pgfxyz(1,1,0)}
\pgflineto{\pgfxyz(1,-1,0)}
\pgflineto{\pgfxyz(0,-1,1)}
\pgflineto{\pgfxyz(-1,0,1)}
\pgffill
\end{pgfscope}
% lati nascosti
\begin{pgfscope}
\color[rgb]{\hiddenColor}
\pgfmoveto{\pgfxyz(-1,1,0)}
\pgflineto{\pgfxyz(-1,0,-1)}
\pgflineto{\pgfxyz(0,-1,-1)}
\pgflineto{\pgfxyz(-1,-1,0)}
\pgflineto{\pgfxyz(-1,0,-1)}
\pgflineto{\pgfxyz(0,1,-1)}
\pgfstroke
\pgfmoveto{\pgfxyz(-1,0,1)}
\pgflineto{\pgfxyz(-1,-1,0)}
\pgflineto{\pgfxyz(0,-1,1)}
\pgfstroke
\pgfmoveto{\pgfxyz(1,0,-1)}
\pgflineto{\pgfxyz(0,-1,-1)}
\pgflineto{\pgfxyz(1,-1,0)}
\pgfstroke
\end{pgfscope}
% spigoli del cono positivo
\begin{pgfscope}
\pgfsetdash{{0.2cm}{0.1cm}}{0cm}
\pgfline{\pgfxyz(0,0,0)}{\pgfxyz(2,2,0)}
\pgfline{\pgfxyz(0,0,0)}{\pgfxyz(0,1.75,0)}
\pgfline{\pgfxyz(0,0,0)}{\pgfxyz(-1.25,0,1.25)}
\pgfline{\pgfxyz(0,0,0)}{\pgfxyz(0,-1.25,1.25)}
\pgfline{\pgfxyz(0,0,0)}{\pgfxyz(1.25,-1.25,0)}
\end{pgfscope}
% lati visibili
\pgfmoveto{\pgfxyz(1,1,0)}
\pgflineto{\pgfxyz(0,1,-1)}
\pgflineto{\pgfxyz(-1,1,0)}
\pgflineto{\pgfxyz(0,1,1)}
\pgflineto{\pgfxyz(1,1,0)}
\pgflineto{\pgfxyz(1,0,-1)}
\pgflineto{\pgfxyz(1,-1,0)}
\pgflineto{\pgfxyz(0,-1,1)}
\pgflineto{\pgfxyz(1,0,1)}
\pgflineto{\pgfxyz(0,1,1)}
\pgflineto{\pgfxyz(-1,0,1)}
\pgflineto{\pgfxyz(0,-1,1)}
\pgfstroke
\pgfmoveto{\pgfxyz(1,1,0)}
\pgflineto{\pgfxyz(1,0,1)}
\pgflineto{\pgfxyz(1,-1,0)}
\pgfstroke
\pgfline{\pgfxyz(1,0,-1)}{\pgfxyz(0,1,-1)}
\pgfline{\pgfxyz(-1,1,0)}{\pgfxyz(-1,0,1)}
% radici positive
\begin{pgfscope}
\color[rgb]{0.5,0.5,0.5}
\pgfcircle[fill]{\pgfxyz(1,1,0)}{\radius}
\pgfcircle[fill]{\pgfxyz(0,1,1)}{\radius}
\pgfcircle[fill]{\pgfxyz(1,0,1)}{\radius}
\pgfcircle[fill]{\pgfxyz(1,-1,0)}{\radius}
\pgfcircle[fill]{\pgfxyz(0,-1,1)}{\radius}
\pgfcircle[fill]{\pgfxyz(-1,0,1)}{\radius}
\pgfcircle[fill]{\pgfxyz(1,0,0)}{\radius}
\pgfcircle[fill]{\pgfxyz(0,1,0)}{\radius}
\pgfcircle[fill]{\pgfxyz(0,0,1)}{\radius}
\end{pgfscope}
% testo
\pgfputat{\pgfxyz(1,-1.0,1)}{\pgfbox[center,center]{Type $\typeB_3$}}
\end{pgfmagnify}
\end{pgftranslate}
% C3
% trasla
\begin{pgftranslate}{\pgfpoint{13.75cm}{2.5cm}}
% scala
\begin{pgfmagnify}{0.7}{0.7}
% setta gi assi
\pgfsetxvec{\pgfpoint{1.374cm}{-0.569cm}}
\pgfsetyvec{\pgfpoint{0cm}{1.5cm}}
\pgfsetzvec{\pgfpoint{-1.1795cm}{-0.9055cm}}
\begin{pgfscope}
\pgfsetyvec{\pgfpoint{0cm}{2.3cm}}
% colora la parte positiva del politopo
\begin{pgfscope}
\color{lightgray}
\color[rgb]{\positiveColor}
\pgfmoveto{\pgfxyz(1,0,0)}
\pgflineto{\pgfxyz(0,1,0)}
\pgflineto{\pgfxyz(-1,0,1)}
\pgflineto{\pgfxyz(-0.5,-0.5,0.5)}
\pgflineto{\pgfxyz(0.5,-0.5,0.5)}
\pgffill
\end{pgfscope}
% lati nascosti
\begin{pgfscope}
\color[rgb]{\hiddenColor}
\pgfmoveto{\pgfxyz(1,0,-1)}
\pgflineto{\pgfxyz(-1,0,-1)}
\pgflineto{\pgfxyz(0,1,0)}
\pgfmoveto{\pgfxyz(-1,0,-1)}
\pgflineto{\pgfxyz(-1,0,1)}
\pgfmoveto{\pgfxyz(-1,0,-1)}
\pgflineto{\pgfxyz(0,-1,0)}
\pgfstroke
\end{pgfscope}
% spigoli del cono positivo
\begin{pgfscope}
\pgfsetdash{{0.2cm}{0.1cm}}{0cm}
\pgfline{\pgfxyz(0,0,0)}{\pgfxyz(1.50,0,0)}
\pgfline{\pgfxyz(0,0,0)}{\pgfxyz(0,1.25,0)}
\pgfline{\pgfxyz(0,0,0)}{\pgfxyz(-1.25,0,1.25)}
\pgfline{\pgfxyz(0,0,0)}{\pgfxyz(-0.7,-0.7,0.7)}
\pgfline{\pgfxyz(0,0,0)}{\pgfxyz(0.75,-0.75,0.75)}
\end{pgfscope}
% lati visibili
\pgfline{\pgfxyz(1,0,1)}{\pgfxyz(1,0,-1)}
\pgfline{\pgfxyz(1,0,1)}{\pgfxyz(0,1,0)}
\pgfline{\pgfxyz(1,0,1)}{\pgfxyz(0,-1,0)}
\pgfline{\pgfxyz(1,0,1)}{\pgfxyz(-1,0,1)}
\pgfline{\pgfxyz(-1,0,1)}{\pgfxyz(0,1,0)}
\pgfline{\pgfxyz(-1,0,1)}{\pgfxyz(0,-1,0)}
\pgfline{\pgfxyz(1,0,-1)}{\pgfxyz(0,1,0)}
\pgfline{\pgfxyz(1,0,-1)}{\pgfxyz(0,-1,0)}
% radici positive
\begin{pgfscope}
\color[rgb]{0.5,0.5,0.5}
\pgfcircle[fill]{\pgfxyz(1,0,1)}{\radius}
\pgfcircle[fill]{\pgfxyz(1,0,0)}{\radius}
\pgfcircle[fill]{\pgfxyz(0,1,0)}{\radius}
\pgfcircle[fill]{\pgfxyz(0.5,0.5,0.5)}{\radius}
\pgfcircle[fill]{\pgfxyz(0.5,-0.5,0.5)}{\radius}
\pgfcircle[fill]{\pgfxyz(0,0,1)}{\radius}
\pgfcircle[fill]{\pgfxyz(-1,0,1)}{\radius}
\pgfcircle[fill]{\pgfxyz(-0.5,-0.5,0.5)}{\radius}
\pgfcircle[fill]{\pgfxyz(-0.5,0.5,0.5)}{\radius}
\end{pgfscope}
% testo
\end{pgfscope}
\pgfputat{\pgfxyz(1,-1.0,1)}{\pgfbox[center,center]{Type $\typeC_3$}}
\end{pgfmagnify}
\end{pgftranslate}
\end{pgfpicture}
\caption{The rank 3 root polytopes}
\label{figure_rankThreeRootPolytopes}
\end{figure}

Notice that for $\typeA_3$ and $\typeC_3$, $\Phi^+$ is \polyhedral\ while it is not for $\typeB_3$ as one may see in Figure \ref{figure_rankThreeRootPolytopes}. Moreveor it is clear that $\Phi^+$ is not \polyhedral\ for $\typeG_2$ (see for example the tables in \cite{bourbaki}) and it is easy to show that $\Phi^+$ is not \polyhedral\ for $\typeD_4$; further if $\Phi^+$ is \polyhedral\ then also the set of positive roots of a subsystem is \polyhedral. Hence only $\typeA_\ell$ and $\typeC_\ell$ may have a \polyhedral\ positive root set and, indeed, this is proved in \cite{celliniMarietti2}.

So $\Phi^+$ is polyhedral if and only if the positive length map coincides with the length map. This suggests that some result similar to the formula in Theorem \ref{introTheorem_formula} should hold also for $\gamma\longmapsto|\gamma|_+$ using a half--space presentation of $\RP_\Phi^+$.

The paper is organized as follows. In Section \ref{section_rootSystem} we fix some notation for the root systems.

In Section \ref{section_rootPolytope} we review the main results of \cite{celliniMarietti} relevant to our aims. In Section \ref{section_preliminary} we see some preliminary result about faces and subsystems. In Section \ref{section_faceInclusion} we study the face inclusion relation for the root polytope, in particular we describe the pair of adjacent facets, i.e. of facets having maximal intersection. In Section \ref{section_normality} we compute the generators for the monoid $M(F)$ and we study the integral closure and normality property. In this section we prove Theorem \ref{introTheorem_properMinimalElements} and Corollary \ref{introCorollary_faceNormality} and \ref{introCorollary_faceIntegrality}. Finally in Section \ref{section_applicationLength} we prove Theorem \ref{introTheorem_formula}, we give a direct formula for length and positive length in type $\typeA_\ell$ and we prove Corollary \ref{introCorollary_lengths}.

\vskip 0.5cm

\noindent{\bf Acknowledgments.} I would like to thank Mario Marietti, Paola Cellini and Andrea Maffei for useful conversations. I thank also Giovanni Gaiffi and Francesco Brenti for helpful literature references. Finally I gratefully recognize the suggestions of an anonymous referee which led to significant improvements.

\section{Root system notation}\label{section_rootSystem}

Let $\Phi$ be a crystallographic reduced root system of rank $\ell$ in the Euclidean space $E$ whose scalar product is denoted by $(u,v)$ for $u,\,v\in E$. Let also $\Phi\supset\Phi^+\supset\Delta$ be a positive subsystem of $\Phi$ and $\Delta$ the corresponding basis, moreover let $R\doteq\langle\Delta\rangle_\Z$ be the lattice generated by the roots.

For a simple root $\alpha$ let $\coomega_\alpha$ be the corresponding fundamental coweight; so $(\coomega_\alpha,\beta)=\delta_{\alpha,\beta}$ for $\alpha,\,\beta\in\Delta$, or, in other words, the coweights $\coomega_\alpha$ with $\alpha\in\Delta$ are the dual basis of $\Delta$. We denote by $\coalpha$ the dual root $2\alpha/(\alpha,\alpha)$ of a root $\alpha$; further $\omega_\alpha$, $\alpha\in\Delta$, is the dual basis of $\coalpha$, $\alpha\in\Delta$. The scalar product of $E$ is normalized so that, for simply laced systems, $\alpha=\coalpha$ (and so also $\omega_\alpha=\coomega_\alpha$) for all $\alpha\in\Delta$.

Given a subset $A$ of $\Delta$ let $W_A$ be the parabolic subgroup of the Weyl group $W$ of $\Phi$ generated by the simple reflections $s_\alpha$ for $\alpha\in A$. We denote by $W^A$ the set of minimal length representatives of $W/W_A$. For a dominant coweight $\lambda$ we define $W_\lambda$ and $W^\lambda$ as the stabilizer, respectively, the set of minimal length representatives of the quotient $W/W_\lambda$. Clearly $W_\lambda=W_A$ with $A$ the set of all simpe roots $\alpha$ such that $(\lambda,\alpha)=0$.

For an irreducible root system we denote by $\theta$ the highest root of $\Phi$ with respect to $\Delta$ and we define $m_\alpha\in\N$, $\alpha\in\Delta$,  such that $\theta=\sum_{\alpha\in\Delta}m_\alpha\alpha$. If $\Phi$ is not simply laced then $\theta_s$ is the short highest root; we will also write $\theta_l$ for $\theta$. Finally let $\Phi_l$ be the set of long roots and $\Phi_s$ the set of short roots; if $\Phi$ is simply laced then we consider all roots as long.

Any time we need to number the elements of the basis $\Delta$ we use the numbering in \cite{bourbaki}; further all symbols indexed by simple roots are accordingly numbered.

\section{The root polytope}\label{section_rootPolytope}

We define the \emph{root polytope} $\RP_\Phi$ of $\Phi$ as the convex hull in $E$ of $\Phi$.

Given an element $\lambda\in E$ let $H(\lambda)$ be the closed half--space $\{u\in E\,|\,(\lambda,u)\leq 1\}$. Since $\Phi$ is invariant for $E\ni u\longmapsto-u\in E$ and it spans $E$, $0$ is an interior point of $\RP_\Phi$. So, being $\Phi$ also a finite set, there exist elements $\lambda_1,\lambda_2,\ldots,\lambda_r$ such that $\RP_\Phi=\bigcap_{i=1}^rH(\lambda_i)$; this is a \emph{half--space} presentation of $\RP_\Phi$. A \emph{face} $F$ of $\RP_\Phi$ is the set of elements $u\in\RP_\Phi$ such that $(\lambda_i,u)=1$ for $i\in A\subseteq\{1,2,\ldots,r\}$ for some subset $A$; we say that the elements $\lambda_i,\,i\in A$, \emph{define} the face $F$. We denote by $V(F)$ the set $F\cap\Phi$ of roots in $F$; it contains the vertices of $F$ but it is in general a larger set.

Now we recall the main results of \cite{celliniMarietti} relevant to our aims. We assume for the rest of this section that $\Phi$ is irreducible.

The affine root system $\affPhi$ may be defined by adding a suitable node to the Dynkin diagram of $\Phi$ (see \cite{bourbaki}) and a simple root $\alpha_0$ to $\Delta$ obtaining the basis $\affDelta=\Delta\cup\{\alpha_0\}$ of the affine root system. The scalar product of $E$ is extended to a bilinear form by declaring $(\alpha_0,\alpha)=-(\theta,\alpha)$ for all $\alpha\in\Delta$.

Notice that $(\frac{\coomega_\alpha}{m_\alpha},\beta)\leq 1$ for all $\beta\in\Phi$. So, given a subset $A$ of $\Delta$, the set $F(A)$ of the elements $u\in\RP_\Phi$ such that $(\frac{\coomega_\alpha}{m_\alpha},u)=1$ for all $\alpha\in A$ is a face of $\RP_\Phi$; we call it a \emph{standard parabolic} face. Let also $V(A)\doteq V(F(A))$.

The face $F(A)$ is clearly the intersections of the faces $F(\alpha)\doteq F(\{\alpha\})$ for $\alpha\in A$; we call such a face $F(\alpha)$ the \emph{coordinate} face associated with the simple root $\alpha$. We set also $V(\alpha)\doteq V(F(\alpha))$ for short.

In general the map $A\longmapsto F(A)$ from subsets of $\Delta$ to standard parabolic faces fails to be injective. As a first result we want to see how a fixed face may be defined in terms of subsets $A$; so we introduce the following definitions. Given $A\subseteq\Delta$ define $\overline{A}$ as the complement in $\affDelta$ of the connected component of $(\Delta\setminus A)\cup\{\alpha_0\}$ containing $\alpha_0$ and define $\partial A$ as the set of the simple roots $\alpha\in A$ such that there exists $\beta$ in the connected component of $(\Delta\setminus A)\cup\{\alpha_0\}$ containing $\alpha_0$ with $(\alpha,\beta)\neq0$. We define also $A^*\doteq\Delta\setminus\partial A$. It is clear that $\partial A\subseteq A\subseteq\overline{A}$. Moreover

\begin{proposition}\label{proposition_minMaxParabolic}\rm{[Proposition 5.9 in \cite{celliniMarietti}]} $F(B)=F(A)$ if and only if $\partial A\subseteq B\subseteq\overline{A}$.
\end{proposition}

Let $\IRR$ be the set of subsets $A$ of $\Delta$ such that $(\Delta\setminus A)\cup\{\alpha_0\}$ is connected in the Dynking diagram of $\affPhi$. By \ref{proposition_minMaxParabolic} the map $\IRR\ni A\longmapsto F(A)$ is a bijection from $\IRR$ to the set of standard parabolic faces; moreover if $A\supseteq B$ then $F(A)\subseteq F(B)$.

It is clear that not all faces of $\RP_\Phi$ are standard parabolic. Since $W$ acts on the faces of $\RP_\Phi$, we consider the faces $F(A;\tau)\doteq\tau F(A)$, for $A$ a subset of $\Delta$ and $\tau\in W$, that we call \emph{parabolic} faces and we define analogously $V(A;\tau)\doteq F(A;\tau)\cap\Phi$.

\begin{proposition}\label{proposition_faces}\rm{[Lemma 5.5 and Propositions 5.10, 5.11 in \cite{celliniMarietti}]} The orbits of $W$ on the set of faces of $\RP_\Phi$ are in bijection with $\IRR$; so any face is parabolic. The codimension of the face $F(A)$, $A\in\IRR$, is $|A|$ and its stabilizer in $W$ is $W_{A^*}$. Hence the faces of the root polytope are $F(A;\tau)$ with $A\in\IRR$ and $\tau\in W^{A^*}$.
\end{proposition}

In particular the facets (i.e. the faces of maximal dimension, equivalently those spanning an affine subspace of $E$ of dimension $\ell-1$) of $\RP_\Phi$ are in the orbits of the coordinate faces $F(\alpha)$ with $\alpha$ simple root such that $(\Delta\setminus\{\alpha\})\cup\{\alpha_0\}$ is connected; we call such simple roots \emph{maximal}. Moreover we find at once the following half--space description of the root polytope.

\begin{proposition}\label{proposition_halfSpaces}\rm{[Corollary 5.13 in \cite{celliniMarietti}]} $\RP_\Phi=\cap H(\frac{\tau\coomega_\alpha}{m_\alpha})$ where $\alpha$ runs in the set of maximal roots and $\tau\in W^{\Delta\setminus\{\alpha\}}$.
\end{proposition}

\begin{remark}\label{remark_betaA}
Let us recall for later use that, by Corollary 5.8 and Proposition 5.4 in \cite{celliniMarietti}, there exists a long root $\beta_A\in V(A)$ such that $V(A)$ is the set of the roots $\beta\in\Phi$ such that $\beta\geq\beta_A$, where $\geq$ is the dominant order; moreover $(\coomega_\alpha,\beta_A)=m_\alpha$ for all $\alpha\in\overline{A}$ and $(\coomega_\alpha,\beta_A)<m_\alpha$ for all $\alpha\not\in\overline{A}$.
\end{remark}

\section{Preliminary results about faces and subsystems}\label{section_preliminary}

In this section we introduce some notation and see various preliminary results needed in the subsequent sections.

Let $F$ be a face of the root polytope $\RP_\Phi$ of the irreducible root system $\Phi$. We define $N(F)$ as the monoid generated by $V(F)$, $Z(F)$ as the lattice generated by $V(F)$, $M(F)$ as the intersection $C(V(F))\cap R$, where $C(V(F))$ is the set of all non--negative rational linear combinations of $V(F)$. We will write also $N(V(F))$ for $N(F)$ and $Z(V(F))$ for $Z(F)$. In order to compare the two monoids $M(F)$ and $N(F)$ we define the following relation on $M(F)$: $v\leq_F u$ if $u-v\in N(F)$.

\begin{lemma}\label{lemma_order} The relation $\leq_F$ is an order on $M(F)$. Moreover $M(F)$ is the union of $\gamma+N(F)$ where $\gamma$ runs in the set of $\leq_F$--minimal elements.
\end{lemma}
\begin{proof} The relation is clearly transitive since $N(V)$ is a monoid. If $u\leq_F v$ and $v\leq_F u$ then $u-v\in N(F)\cap(-N(F))$, but only $0$ belongs to such intersection since $N(F)$ is strongly convex being $V(F)$ defined as the $1$--level set on $\Phi$ of some functionals on $E$; so $u=v$.

Now let $\lambda$ be an element of $E$ appearing in a half--space presentation of $\RP_\Phi$ such that $(\lambda,w)=1$ for all $w\in F$. In particular if $u\leq_F v$, $u\neq v$ we have $(\lambda,u)\leq(\lambda,v)-1$. Hence any strictly descending chain of elements of $M(F)$ must be finite since for all $u\in M(F)$ we have $(\lambda,u)\geq 0$. The last claim is now clear.
\end{proof}

In the following lemmas we relate the faces of the root polytope and the root subsystems. We sometime add the root system symbol as a subscript for clarity.

\begin{lemma}\label{lemma_subfaceSubsystem}
Let $F\doteq F(A)$ be a face of the root polytope with $A\in\IRR$, let $E'$ be the vector subspace of $E$ generated by $F$ and let $\Phi'$ be the root subsystem $\Phi\cap E'$ of $\Phi$. Then there exists $\gamma_A\in\Phi$ such that
\begin{itemize}
\item[(i)] $(\Delta\setminus A)\cup\{\gamma_A\}$ is a basis of $\Phi'$,
\item[(ii)] $F=F_{\Phi'}(\{\gamma_A\})$.
\end{itemize}
\end{lemma}
\begin{proof}
\begin{itemize}
\item[(i)] Let $\Phi'^+\doteq\Phi^+\cap E'$; it is a positive subsystem of $\Phi'$. The basis of $\Phi'$ corresponding to $\Phi'^+$ contains $\Delta\setminus A$ since $\Delta\setminus A$ is a subset of the basis $\Delta$ of $\Phi$ corresponding to $\Phi^+$.

By Proposition \ref{proposition_faces} the dimension of $E'$ is $|\Delta|-|A|+1$; hence there exists an uniquely determined $\gamma_A\in\Phi'^+$ such that $(\Delta\setminus A)\cup\{\gamma_A\}$ is the basis of $\Phi'$ corresponding to $\Phi'^+$.
\item[(ii)] Notice that the subsystem $\Phi'$ is irreducible by (1) of Corollary 4.5 in \cite{celliniMarietti}. So let $\theta'$ be the highest root of $\Phi'$ with respect to $\Phi'^+$ and let $\theta'=\sum_{\delta\in\Delta\setminus A}m'_\delta\delta+m'\gamma_A$ for some non--negative integers $m'_\delta$, $\delta\in\Delta\setminus A$ and $m'$. Further $\beta_A\in\Phi'^+$ and so we may write $\beta_A=\sum_{\delta\in\Delta\setminus A}a_\delta\delta+a\gamma_A$ for some non--negative integers $a_\delta$, $\delta\in\Delta\setminus A$, and $a$. We claim that $a=m'$.

Indeed suppose that $\gamma_A=\sum_{\alpha\in\Delta}c_\alpha\alpha$ for some non--negative integers $a_\alpha$, $\alpha\in\Delta$. Then $\beta_A=\sum_{\delta\in\Delta\setminus A}(a_\delta+c_\delta)\delta+\sum_{\alpha\in A}(ac_\alpha)\alpha$ and we find that $ac_\alpha=m_\alpha$ for any $\alpha\in A$ since $(\coomega_\alpha,\beta_A)=m_\alpha$ for any $\alpha\in A$. (We want to stress that here and in the remaing of this proof, a coweight $\coomega_\alpha$ with $\alpha\in A$ is always related to $\Phi$ and \emph{not} to $\Phi'$; i.e. $\coomega_\alpha$ is an element of the dual basis of the basis $\Delta$ of $\Phi$.)

Now notice that $\theta\geq\theta'=\sum_{\delta\in\Delta\setminus A}(m'_\delta+m'c_\delta)\delta+\sum_{\alpha\in A}(m'c_\alpha)\alpha$. Hence $m'c_\alpha\leq m_\alpha$ for any $\alpha\in A$. So $m'\leq a$ using $ac_\alpha=m_\alpha$.

On the other hand $\beta_A\leq\theta'$ clearly implies that $a\leq m'$ and so we have proved our claim that $a=m'$.

We may now easily conclude the proof of the lemma. Indeed since $a=m'$ we have $\beta_A\in F_{\Phi'}(\{\gamma_A\})$, hence $F\subseteq F_{\Phi'}(\{\gamma_A\})$ since $F$ is the set of all roots $\beta\in\Phi$ such that $\beta\geq\beta_A$.

Finally let $\beta\in F_{\Phi'}(\{\gamma_A\})$ and $\alpha\in A$. We have $(\coomega_\alpha,\beta)=(\coomega_\alpha,m'\gamma_A)=(\coomega_\alpha,a\gamma_A)=ac_\alpha=m_\alpha$ and so $\beta\in F$.
\end{itemize}
\end{proof}

\begin{lemma}\label{lemma_subface} If $F'$ is a subface of the face $F$ of $\RP_\Phi$ and $E'$ is the subspace of $E$ spanned by $F'$ then $F'$ is a facet of $\RP_{\Phi'}$ and $V_\Phi(F)\cap E'=V_{\Phi'}(F')$ where $\Phi'$ is the root subsystem $\Phi\cap E'$ of $\Phi$.
\end{lemma}
\begin{proof}
By the previous Lemma \ref{lemma_subfaceSubsystem}, $F'$ is a facet of $\RP_{\Phi'}$ and then the claim is clear since $V_\Phi(F)\cap E' = F\cap E'\cap\Phi = F'\cap\Phi'=V_{\Phi'}(F')$.
\end{proof}

Now we proceed with a lemma describing the lattice generated by the roots in a face.

\begin{lemma}\label{lemma_faceLattice}
Let $A\in\IRR$, then $Z(V(A))=\langle\Delta\setminus A,\beta_A\rangle_\Z$.
\end{lemma}
\begin{proof} If $\beta\in V(A)$ then $(\coomega_\epsilon,\beta)=m_\epsilon$ for all $\epsilon\in A$, hence $\beta=\beta_A+\sum_{\alpha\in\Delta\setminus A}c_\alpha\alpha$ with $c_\alpha\in\N$ for all $\alpha\in\Delta\setminus A$; so $\beta\in\langle\Delta\setminus A,\beta_A\rangle_\Z$.

In order to prove the reverse inclusion consider a maximal chain $e=\tau_0<\tau_1<\cdots<\tau_r$ in $W^\theta$ with respect to the left weak Bruhat order such that $\tau_r\theta=\beta_A$ and let $\alpha_1,\alpha_2,\ldots,\alpha_r$ be simple roots such that $\tau_i=s_{\alpha_i}\tau_{i-1}$ for $i=1,2,\ldots,r$. Since $\tau_i\theta\geq\beta_A$ we have $\tau_i\beta\in V(A)$ and so $\{\alpha_1,\alpha_2,\ldots,\alpha_r\}\subseteq\Delta\setminus A$; further this is a set equality since $(\coomega_\alpha,\beta_A)<m_\alpha$ for all $\alpha\in\Delta\setminus A$.

We have proved that the two roots $\tau_{i-1}\theta$ and $\tau_i\theta=\tau_{i-1}\theta+a\alpha_i$, for some positive $a\in\N$, in the $\alpha_i$--string through $\tau_{i-1}\theta$ are in $V(A)$; but a root string is unbroken, hence also $\tau_i\theta+\alpha_i\in V(A)$. We conclude that $\Delta\setminus A\subset Z(V(A))$. This finishes our proof since it is clear that $\beta_A\in Z(V(A))$.
\end{proof}

Finally in the next lemma we see how the minimal elements of the monoid in a face for non--reduced root system may be described in terms of the irreducible components.
\begin{lemma}\label{lemma_reducible} Suppose $E=E_1\oplus E_2$, where $\oplus$ is the orthogonal direct sum, and $\Phi_i\subset E_i$, for $i=1,2$, are root systems. Suppose $F$ is a face of $\RP_{\Phi_1\cup\Phi_2}$ and let $F_i\doteq F\cap E_i$ for $i=1,\,2$.
\begin{itemize}
\item[(i)] $F_1$ and $F_2$ are faces of $\RP_{\Phi_1}$, respectively, of $\RP_{\Phi_2}$.
\item[(ii)] $M(F)=M(F_1)\oplus M(F_2)$ and $N(F)=N(F_1)\oplus N(F_2)$
\item[(iii)] The set of non--zero $\leq_F$--minimal elements of $M(F)$ is the union of those of $M(F_1)$ and of $M(F_2)$
\end{itemize}
\end{lemma}
\begin{proof} Denote by $\Phi$ the root system $\Phi_1\cup\Phi_2$.
\begin{itemize}
\item[(i)] Let $\lambda_1,\lambda_2,\ldots,\lambda_r\in E$ give a half--space presentation of $\RP_\Phi$ and suppose that $A\subseteq\{1,2,\ldots,r\}$ is such that the elements $\lambda_h$, $h\in A$ define the face $F$. Since $\RP_{\Phi_1}=\RP_\Phi\cap E_1$, we see that, writing $\lambda_h=(\lambda_{h,1},\lambda_{h,2})$, for $h=1,2,\ldots,r$, the vectors $\lambda_{1,1},\lambda_{2,1},\ldots,\lambda_{r,1}$ give a half--space presentation of $\RP_{\Phi_1}$. In particular, the vectors $\lambda_{h,1}$, $h\in A$ define $F_1=F\cap E_1$, hence $F_1$ is a face of $\RP_{\Phi_1}$. The proof for $F_2$ is analogous.
\item[(ii)] It is clear that $V(F)=V(F_1)\cup V(F_2)$ and so $N(F)=N(F_1)\oplus N(F_2)$ follows. Moreover $M(F)=C(V(F))\cap R_\Phi=C(V(F_1))\oplus C(V(F_2))\cap R_{\Phi_1}\oplus R_{\Phi_2}=(C(V(F_1))\cap R_{\Phi_1})\oplus(C(V(F_2))\cap R_{\Phi_2})=M(F_1)\oplus M(F_2)$.
\item[(iii)] This follows at once by (ii).
\end{itemize}
\end{proof}

\section{The face inclusion relation}\label{section_faceInclusion}

We want to study the inclusion condition for the faces of $\RP_\Phi$. We begin by the following proposition; it is a slightly improved version of Lemma 4.2 in \cite{celliniMarietti}.

\begin{proposition}\label{proposition_barycenter} If $A\in\IRR$ then $W_{A^*}$ is the stabilizer of the barycenter
$$
b(V(A))\doteq\frac{1}{|V(A)|}\sum_{\beta\in V(A)}\beta
$$
of $V(A)$. In particular the $W$--orbit of $F(A)$ is in bijection with the $W$--orbit of $b(V(A))$.
\end{proposition}
\begin{proof} Since $W_{A^*}$ is the stabilizer of $F(A)$, it clearly stabilizes $b(V(A))$. Viceversa $b(V(A))$ is in the dominant chamber by Lemma 4.2 in \cite{celliniMarietti}, so its stabilizer is generated by the simple reflections it contains. Setting $b\doteq b(V(A))$ for short, it suffices to show that $(b,\alpha)>0$ for all $\alpha\in\partial A$ to prove our claim.

Now we show first that $(\beta_A,\alpha)>0$ for all $\alpha\in\partial A$. Indeed let us write $\beta_A=\theta-\sum_{\gamma\in\Delta\setminus A}c_\gamma\gamma$ with $c_\gamma>0$, for all $\gamma\in\Delta\setminus A$, and let $\alpha\in\partial A$. We have $(\beta_A,\alpha)=(\theta,\alpha)-\sum_{\gamma\in\Delta\setminus A}c_\gamma(\gamma,\alpha)$ where $(\theta,\alpha)\geq0$ since the highest root $\theta$ is in the dominant chamber and $(\gamma,\alpha)\leq0$ for all $\gamma\in\Delta\setminus A$, since $\alpha\in A$. This shows that $(\beta_A,\alpha)$ is positive as soon as $(\theta,\alpha)>0$ or $(\gamma,\alpha)\neq0$ for some $\gamma\not\in A$. Hence $(\beta_A,\alpha)>0$, by definition of $\partial A$.

Now consider a general $\beta\in V(A)$. We can write as above $\beta=\theta-\sum_{\gamma\in\Delta\setminus A}c_\gamma\gamma$ but this time $c_\gamma\geq0$. So the same argument as above shows that $(\beta,\alpha)\geq0$ for all $\alpha\in\partial A$. This proves our claim since for $\alpha\in\partial A$ we have $(b,\alpha)=1/|V(A)|\sum_{\beta\in V(A)}(\beta,\alpha)\geq1/|V(A)|(\beta_A,\alpha)>0$.
\end{proof}

\begin{proposition}\label{proposition_faceBorder} If $A,B\in\IRR$ then the face $F(B;\sigma)$ is contained in the face $F(A;\tau)$ if and only if $B\supseteq A$ and $\tau^{-1}\sigma\in W_{A^*}\cdot W_{B^*}$.
\end{proposition}
\begin{proof} If $\tau^{-1}\sigma=\eta_A\eta_B$ with $\eta_A\in W_{A^*}$, $\eta_B\in W_{B^*}$ and $B\supseteq A$ then $F(B;\sigma)=\tau\eta_A\eta_B\cdot F(B)=\tau\eta_A\cdot F(B)\subseteq\tau\eta_A\cdot F(A)=F(A;\tau)$.

In order to prove the converse let $\eta\doteq s_{\alpha_1}s_{\alpha_2}\cdots s_{\alpha_r}$, for some simple roots $\alpha_1,\alpha_2,\cdots,\alpha_r$, be a reduced expression of the minimal representative of $\tau^{-1}\sigma$ in $W^{B^*}$ and let $b\doteq b(V(B))$. So, by (iv) of Theorem~4.3.1 in \cite{brentiBjorner}, we have $b>s_{\alpha_r}b>s_{\alpha_{r-1}}s_{\alpha_r}b>\cdots>s_{\alpha_1}s_{\alpha_2}\cdots s_{\alpha_r}b$, with respect to the dominant order, being the above expression of $\eta$ reduced. Hence, in particular, using Proposition \ref{proposition_barycenter}, there exists $\beta\in s_{\alpha_2}\cdots s_{\alpha_r}\cdot V(B)$ such that $\beta'\doteq s_{\alpha_1}(\beta)<\beta$ and so $(\beta,\coalpha_1)>0$.

But $\beta$ is a root, so $(\coomega_{\alpha_1},\beta)\leq m_{\alpha_1}$ and we find $(\coomega_{\alpha_1},\beta')=(\coomega_{\alpha_1},\beta)-(\beta,\coalpha_1)<m_{\alpha_1}$. We conclude $\alpha_1\not\in A$ since $\beta'\in\eta V(B)\subseteq V(A)$. Hence $\alpha_1\not\in\partial A$, so we have $s_{\alpha_2}\cdots s_{\alpha_r}\cdot F(B)\subseteq s_{\alpha_1}\cdot F(A)=F(A)$. We inductively find that $\eta\in W_{A^*}$, hence $\tau^{-1}\sigma\in W_{A^*}\cdot W_{B^*}$. It follows that $F(B)\subseteq F(A)$ and so finally $B\supseteq A$ since $A,B\in\IRR$.
\end{proof}

As a first application of the previous proposition, in the following lemma we see that the lattice generated by the roots in a face is compatible with subfaces.

\begin{lemma}\label{lemma_subfaceLattice}
Let $F'$ be a subface of the face $F$ of $\RP_\Phi$. Then $Z(V(F'))=Z(V(F))\cap\langle F'\rangle_\R$.
\end{lemma}
\begin{proof}
Using Proposition \ref{proposition_faceBorder} we may assume that $F$ and $F'$ are standard parabolic faces using the action of the Weyl group; so let $F\doteq F(A)$ and $F'\doteq F(A')$ for some $A\subseteq A'\subseteq\Delta$. Further it suffices to prove our claim in the case $A'\doteq A\cup\{\delta\}$ with $\delta\in\Delta\setminus A$. Let $E'\doteq\langle F(A')\rangle_\R$ and, for a generic $B\in\IRR$, let $\beta'_B\doteq\sum_{\alpha\in B}m_\alpha\alpha$; notice that by the previous Lemma \ref{lemma_faceLattice} we have $Z(V(B))=\langle\Delta\setminus B,\beta_B\rangle_\Z=\langle\Delta\setminus B,\beta'_B\rangle_\Z$. Finally let $\alpha$ be an arbitrary fixed element of $A$, denote by $\varphi$ the vector $\frac{\coomega_\alpha}{m_\alpha}-\frac{\coomega_\delta}{m_\delta}$ and by $L$ the set of $u\in E$ such that $(\varphi,u)=0$.

Let $\gamma$ be an element of $Z(V(A))\cap E'=\langle\Delta\setminus A,\beta'_A\rangle_\Z\cap L$. Then $\gamma=c\beta'_A+\sum_{\alpha\in\Delta\setminus A}c_\alpha\alpha$ for some $c\in\N$ and $c_\alpha\in\N$ for all $\alpha\in\Delta\setminus A$, and using the fact that $0=(\varphi,\gamma)$ we find $c_\delta=cm_\delta$. Hence $\gamma=c(\beta'_A+m_\delta\delta)+\sum_{\alpha\in\Delta\setminus A'}c_\alpha\alpha=c\beta'_{A'}+\sum_{\alpha\in\Delta\setminus A'}c_\alpha\alpha$. So $\gamma\in\langle\Delta\setminus A',\beta'_{A'}\rangle_\Z=Z(V(A'))$.

Now our claim is proved since the inclusion $Z(V(A'))\subseteq Z(V(A))\cap\langle F(A')\rangle_\R$ is clear.
\end{proof}

We apply again Proposition \ref{proposition_faceBorder} to the description of the border of a facet $F(\alpha)$, $\alpha\in\Delta$ a maximal root. We say that two faces of the same dimension $d$ are \emph{adjacent} if their intersection is a face of dimension $d-1$.

\begin{proposition}\label{proposition_adjacentFacet} Let $\alpha$ be a maximal root. The facets adjacent to $F(\alpha)$ are:
\begin{itemize}
\item[(i)] $F(\delta;\tau)$ with $\delta\in\Delta$, $\delta\neq\alpha$ maximal root such that $\{\alpha,\delta\}\in\IRR$ and $\tau\in W_{\Delta\setminus\{\alpha\}}$,
\item[(ii)] $F(\alpha;\tau s_\alpha)$ with $\tau\in W_{\Delta\setminus\{\alpha\}}$, if there exists $\epsilon\in\Delta$ non--maximal such that $\{\alpha,\epsilon\}\in\IRR$. In such a case $\epsilon$ is the unique simple root adjacent to $\alpha$ in the Dynkin diagram of $\affPhi$.
\end{itemize}
\end{proposition}
\begin{proof} Let the facet $F_1\doteq F(\delta;\sigma)$, with $\delta\in\Delta$ maximal root, be adjacent to $F\doteq F(\alpha)$ and let $F'\doteq F\cap F_1$. Since $F'$ has codimension $2$ and is contained in $F$, by Proposition \ref{proposition_faces} and Proposition \ref{proposition_faceBorder} we have $F'=F(\{\alpha,\epsilon\};\tau)$ for some $\epsilon\in\Delta$, $\epsilon\neq\alpha$ such that $\{\alpha,\epsilon\}\in\IRR$ and $\tau\in W_{\Delta\setminus\{\alpha\}}$ since this is the stabilizer of $F$.

Using the fact that $F'\subset F_1$ and Proposition \ref{proposition_faceBorder}, we obtain that $\delta\in\{\alpha,\epsilon\}$. If $\epsilon$ is a maximal root then $\delta=\epsilon$ since otherwise if $\delta=\alpha$ then $F'$ would be contained in the three different facets $F$, $F_1$ and $F(\{\epsilon\};\tau)$ and this is clearly impossible. Hence $F_1$ is of type (i) as in our claim.

So suppose that $\epsilon$ is not maximal; then $\delta=\alpha$ since $\delta$ is maximal. Consider the subspace $E'$ of codimensione $1$ in $E$ spanned by $F'$. It is clearly the orthogonal of $\frac{\coomega_\alpha}{m_\alpha}-\frac{\tau\coomega_\epsilon}{m_\epsilon}$ and $\frac{\coomega_\alpha}{m_\alpha}-\frac{\sigma\coomega_\alpha}{m_\alpha}$. Since in turn $E'$ determines $F'$ and hence $F_1$, if we show that $\frac{\coomega_\alpha}{m_\alpha}-\frac{\tau\coomega_\epsilon}{m_\epsilon}$ and $\frac{\coomega_\alpha}{m_\alpha}-\frac{\tau s_\alpha\coomega_\alpha}{m_\alpha}$ are proportional, then we may conclude that $F_1$ is of type (ii) as in the claim.

Now notice that $(\Delta\setminus\{\alpha,\epsilon\})\cup\{\alpha_0\}$ is connected, since $\{\alpha,\epsilon\}\in\IRR$, while, being $\epsilon$ not maximal, $(\Delta\setminus\{\epsilon\})\cup\{\alpha_0\}$ is not connected in the Dynkin diagram of $\affPhi$; hence $\epsilon$ is the unique (simple) root in the Dynkin diagram of $\affPhi$ connected to $\alpha$.

In particular $\coalpha=2\coomega_\alpha-a\coomega_\epsilon$, for some $a\in\N$, and $(\theta,\alpha)=0$; so substituting the first equation in the second one we find $a=2m_\alpha/m_\epsilon$. We compute
$$
\begin{array}{rcl}
\frac{\coomega_\alpha}{m_\alpha}-\frac{\tau s_\alpha\coomega_\alpha}{m_\alpha} & = & \frac{\coomega_\alpha}{m_\alpha}-\tau(-\frac{\coomega_\alpha
}{m_\alpha}+2\frac{\coomega_\epsilon}{m_\epsilon}))\\
 & = & 2(\frac{\coomega_\alpha}{m_\alpha}-\frac{\tau\coomega_\epsilon}{m_\epsilon}).
\end{array}
$$
This finishes the proof that $F_1$ is of type (ii) in our claim when $\epsilon$ is not a maximal root.

Now we need to show that each facet of type (i) and (ii) is adjacent to $F$. But this is clear since: by (i) we have that $F(\{\alpha,\delta\};\tau)$ is a codimension $2$ face in $F(\alpha)\cap F(\delta;\tau)$ while by (ii) $F(\{\alpha,\epsilon\};\tau)$ is a codimension $2$ face in $F(\alpha)\cap F(\alpha;\tau s_\alpha)$.
\end{proof}

If a facet $F$ has some adjacent facet that is in the Weyl group orbit of $F$ then we say that $F$ has \emph{autointersection}; notice that these facets are described in {\it (ii)} of the previous proposition.

We finish this section showing how a facet decomposes in orbits under the action of its stabilizer; this will be used later.

\begin{lemma}\label{lemma_facetOrbits}
Let $\alpha$ be a maximal root. If $\Phi$ is not simply laced and the highest short root $\theta_s\in V(\alpha)$ (i.e. if $(\coomega_\alpha,\theta_s)=m_\alpha$) then $V(\alpha)=(W_{\Delta\setminus\{\alpha\}}\cdot\theta_s)\sqcup(W_{\Delta\setminus\{\alpha\}}\cdot\theta)$ otherwise $V(\alpha)=W_{\Delta\setminus\{\alpha\}}\cdot\theta$.
\end{lemma}
\begin{proof} Clearly $W_{\Delta\setminus\{\alpha\}}\cdot\theta\subset V(\alpha)$ and, if $\Phi$ is not simply laced and $\theta_s\in V(\alpha)$ then  $W_{\Delta\setminus\{\alpha\}}\cdot\theta_s\subset V(\alpha)$.

Conversely, let us suppose $\Phi$ not simply laced, $\theta_s\in V(\alpha)$ and let $\beta$ be a short root in $V(\alpha)$. Clearly $\beta=\tau\theta_s$ for some $\tau\in W$. Notice that $\theta_s$ is a dominant weight, hence its stabilizer $W'$ in $W$ is a standard parabolic subgroup. So let $\tau'=s_{\alpha_1}s_{\alpha_2}\cdots s_{\alpha_r}$, for some $\alpha_1,\alpha_2,\ldots,\alpha_r$, be a reduced expression of the minimal representative of $\tau$ in $W/W'$. By (iv) of Theorem 4.3.1 in \cite{brentiBjorner}, we have $\theta_s-\beta\in\langle\alpha_1,\alpha_2,\ldots,\alpha_r\rangle_\N$ and we conclude $\alpha_i\neq\alpha$ for all $i=1,2,\ldots,r$ since $(\coomega_\alpha,\beta)=m_\alpha=(\coomega_\alpha,\theta_s)$. In particular $\tau'\in W_{\Delta\setminus\{\alpha\}}$ and our claim is proved since $\beta=\tau'\theta_s$.

The proof for the long roots is similar.
\end{proof}

We define $V_l(\alpha)\doteq V(\alpha)\cap\Phi_l=W_{\Delta\setminus\{\alpha\}}\cdot\theta$ and $V_s(\alpha)\doteq V(\alpha)\cap\Phi_s$; this last set is either $W_{\Delta\setminus\{\alpha\}}\cdot\theta_s$ if $\Phi$ is non simply laced and $\theta_s\in V(\alpha)$ or empty otherwise.

\section{Integral closure of the monoids generated by the faces}\label{section_normality}

In this section we study the integral closure of the monoids generated by the faces of the root polytope. In particular we explicitly find the $\leq_F$--minimal elements of $M(F)$, $F$ a face of the root polytope.

First of all notice that, by Lemma \ref{lemma_reducible}, the normality property and the minimal elements for a non--irreducible root system may be determined in terms of the irreducible factors.

Further suppose that an element $\gamma\in M(F)$ is also an element of $M(F')$ with $F'$ a face in the border of $F$, then $F'\subset E'\doteq\langle F'\rangle_\R$ and $E'$ is a proper subspace of $E$. Hence $F'$ is a facet of the root polytope of $\Phi'\doteq\Phi\cap E'$ by Lemma \ref{lemma_subface} and $\Phi'$ is irreducible by (1) of Corollary 4.5 in \cite{celliniMarietti}. 

So throughout this section we will assume that $\Phi$ is irreducible and that $F$ is a facet of $\RP_\Phi$. We define a non--zero element of $M(F)$ to be \emph{proper} if it is not an element of $M(F')$ with $F'$ a face in the border of $F$ and we look for proper minimal elements of $M(F)$.

We want to develop a criterion for proper minimal elements. The minimal elements of $\tau\cdot F$ are the images, under the action of $\tau$, of the minimal elements of $F$. So from now on, we consider a facet $F\doteq F(\alpha)$ for a fixed maximal root $\alpha\in\Delta$. Recall that $F$ is defined by the vector $\lambda\doteq\frac{\coomega_\alpha}{m_\alpha}$ and, in particular, $V\doteq V(F(\alpha))$ is the set of roots $\beta$ such that $(\coomega_\alpha,\beta)=m_\alpha$.

Given $\gamma\in E$ let $\gamma_+$ be the unique element in $W_{\Delta\setminus\{\alpha\}}\cdot\gamma$ that is dominant for $\Delta\setminus\{\alpha\}$ (i.e. such that $(\gamma_+,\delta)\geq0$ for all $\delta\in\Delta$, $\delta\neq\alpha$). Given a simple root $\delta$ and $\tau\in W$ let $\nabla_{\alpha,\delta}^\tau$ be the vector $\frac{\coomega_\alpha}{m_\alpha}-\frac{\tau\coomega_\delta}{m_\delta}$; further we set $\nabla_{\alpha,\delta}\doteq\nabla_{\alpha,\delta}^e$. Finally let $\Psi_\alpha$ be the set of simple roots $\delta$ such that: either $\delta\neq\alpha$ is maximal and $\{\alpha,\delta\}\in\IRR$ (i.e. the complement of $\{\alpha,\delta\}$ is connected in $\affDelta$) or $F$ has autointersection and $\delta$ is the unique root adjacent to $\alpha$ (see (ii) of Proposition \ref{proposition_adjacentFacet}, in this case $\delta$ is not maximal).

\begin{remark}\label{remark_separatingHyperplanes}
By Proposition \ref{proposition_adjacentFacet} the elements $\nabla_{\alpha,\delta}^\tau$, with $\delta\in\Psi_\alpha$ and $\tau\in W_{\Delta\setminus\{\alpha\}}$, define the hyperplanes separating $F$ from its adjacent facets. Indeed this is clear for $\delta$ maximal. Further, if $F$ has autointersection and $\delta$ is the unique root adjacent to $\alpha$, we have $2\nabla_{\alpha,\delta}^\tau=\nabla_{\alpha,\alpha}^{\tau s_\alpha}$ since, as seen in the proof of Proposition \ref{proposition_adjacentFacet}, $\tau s_\alpha\coomega_\alpha=-\coomega_\alpha+2m_\alpha\frac{\tau\coomega_\delta}{m_\delta}$.
\end{remark}

The following simple lemma will be useful in the sequel.
\begin{lemma}\label{lemma_dominantInequality}
Suppose $\gamma$ is $(\Delta\setminus\{\alpha\})$--dominant and let $\tau\in W_{\Delta\setminus\{\alpha\}}$. Then $(\nabla_{\alpha,\delta},\gamma)\leq(\nabla_{\alpha,\delta}^\tau,\gamma)$ for any $\delta\in\Delta$.
\end{lemma}
\begin{proof} By (iv) of Theorem 4.3.1 of \cite{brentiBjorner} we have $\tau^{-1}\gamma=\gamma-\eta$ with $\eta$ a non--negative linear combination of $\Delta\setminus\{\alpha\}$ since $\gamma$ is $(\Delta\setminus\{\alpha\})$--dominant. Using $\tau\in W_{\Delta\setminus\{\alpha\}}$ we get
$$
\begin{array}{rcl}
(\nabla_{\alpha,\delta}^\tau,\gamma) & = & (\frac{\coomega_\alpha}{m_\alpha}-\frac{\tau\coomega_\delta}{m_\delta},\gamma)\\
 & = & (\frac{\coomega_\alpha}{m_\alpha},\gamma)-(\frac{\coomega_\delta}{m_\delta},\tau^{-1}\gamma)\\
 & = & (\frac{\coomega_\alpha}{m_\alpha},\gamma)-(\frac{\coomega_\delta}{m_\delta},\gamma-\eta)\\
 & = & (\nabla_{\alpha,\delta},\gamma)+(\frac{\coomega_\delta}{m_\delta},\eta)
\end{array}
$$
and our claim follows since $(\frac{\coomega_\delta}{m_\delta},\eta)\geq0$.
\end{proof}

\begin{proposition}\label{proposition_criterionProper} An element $\gamma\in R$ is a proper element of $M(F)$ if and only if: $0<(\nabla_{\alpha,\delta},\gamma_+)$ for any $\delta\in\Psi_\alpha$.
\end{proposition}
\begin{proof} First of all notice that the set of proper elements of $M(F)$ is clearly stable by the action of $W_{\Delta\setminus\{\alpha\}}$ since this is the stabilizer of $F$. So $\gamma$ is proper if and only if $\gamma_+$ is proper.

The element $\gamma_+$ is a proper element of $M(F)$ if it is not an element of a subface of $F$, hence if and only if it is in $M(F)$ and not in any hyperplane separating $F$ from a facet adjacent to $F$. So if a facet $F'$ adjacent to $F$ is defined by a vector $\mu$ then we must have $(\lambda-\mu,\gamma_+)>0$ since $\lambda-\mu=0$ is the hyperplane in $E$ separating $F$ and $F'$ and $(\lambda-\mu,\beta)=1-(\mu,\beta)\geq 0$ for all $\beta\in V$.

As seen above these hyperplanes are defined by $\nabla_{\alpha,\delta}^\tau$ for $\delta\in\Psi_\alpha$, $\tau\in W_{\Delta\setminus\{\alpha\}}$; so we have proved that $0<(\nabla_{\alpha,\delta}^\tau,\gamma_+)$ for all $\delta\in\Psi_\alpha$, $\tau\in W_{\Delta\setminus\{\alpha\}}$. In particular $0<(\nabla_{\alpha,\delta},\gamma_+)$ for any $\delta\in\Psi_\alpha$.

Conversely $(\nabla_{\alpha,\delta}^\tau,\gamma_+)\geq(\nabla_{\alpha,\delta},\gamma_+)$ for any $\delta\in\Psi_\alpha$, $\tau\in W_{\Delta\setminus\{\alpha\}}$ by Lemma \ref{lemma_dominantInequality}. This proves that $\gamma_+$ is proper if and only if $0<(\nabla_{\alpha,\delta},\gamma_+)$ for any $\delta\in\Psi_\alpha$.
\end{proof}

Our next result is a criterion for minimal elements of $M(F)$. Given $\delta\in\Psi_\alpha$ let
$$
\begin{array}{rcl}
D_l(\delta) & \doteq & \max\limits_{\beta\in V\cap\Phi_l}(\nabla_{\alpha,\delta},\beta)\\
\\
D_s(\delta) & \doteq & \left\{
\begin{array}{ll}
\max\limits_{\beta\in V\cap\Phi_s}(\nabla_{\alpha,\delta},\beta) & \textrm{if }V\cap\Phi_s\neq\varnothing,\\
+\infty & \textrm{if }V\cap\Phi_s=\varnothing.
\end{array}
\right.
\end{array}
$$

\begin{proposition}\label{proposition_criterionMinimal} A proper element $\gamma$ of $M(F)$ is minimal if and only if: for each $\beta\in V$ there exists $\delta\in\Psi_\alpha$ and $\tau\in W_{\Delta\setminus\{\alpha\}}$, depending on $\beta$, such that $0<(\nabla_{\alpha,\delta}^\tau,\gamma)<(\nabla_{\alpha,\delta}^\tau,\beta)$.

In particular, if there exists a proper minimal element for $M(F)$ then there exist $\delta_l,\delta_s\in\Psi_\alpha$ such that
\begin{equation}\label{equation_criMinL}
\frac{1}{\lcm(m_\alpha,m_{\delta_l})}<D_l(\delta_l)
\end{equation}
\begin{equation}\label{equation_criMinS}
\frac{1}{lcm(m_\alpha,m_{\delta_s})}<D_s(\delta_s)
\end{equation}
\end{proposition}
\begin{proof}
Suppose that $\gamma\in M(F)$ is a proper minimal element and $\beta\in V$. Since $\gamma-\beta\not\in M(F)$ there exists a facet $F'$ adjacent to $F$, defined by a vector $\mu$, such that $\gamma$ and $\gamma-\beta$ are in different half--spaces defined by $\lambda-\mu$. In particular:
\begin{itemize}
\item[(i)] $(\lambda-\mu,\gamma)>0$ since $\gamma$ is a proper element of $M(F)$,
\item[(ii)] $(\lambda-\mu,\gamma-\beta)<0$ since $\gamma-\beta\not\in M(F)$.
\end{itemize}
These two inequalities are clearly equivalent to our claim by the description of the hyperplanes separating $F$ from its adjacent facets given in Remark \ref{remark_separatingHyperplanes}.

In order to prove the converse let $\gamma\in M(F)$ be proper and fulfil the above condition. We have to show that for each $\eta=\sum_{\beta\in V} c_\beta\beta$, $c_\beta\in\N$, the element $\gamma-\eta$ is not in $M(F)$ unless $\eta=0$. Let $\beta_0$ be such that $c_{\beta_0}>0$ and let $\mu\doteq\mu_{\beta_0}$ be such that $0<(\lambda-\mu,\gamma)<1-(\mu,\beta_0)$ as in the condition. We have $(\lambda-\mu,\eta)=\sum_{\beta\in V}c_\beta(1-(\mu,\beta))\geq 1-(\mu,\beta_0)>(\lambda-\mu,\gamma)$, where we have used that $c_\beta\geq0$, $1-(\mu,\beta)\geq0$ for all $\beta\in V$ and $c_{\beta_0}\geq 1$.

So we have $(\lambda-\mu,\gamma-\eta)<0$ and this shows that $\gamma-\eta\not\in M(F)$ since for all elements $\varphi$ of $M(F)$ we have $(\lambda-\mu,\varphi)\geq0$ using the fact that $\mu$ defines a facet adjacent to $F$.

Now we prove the last claim. We see the proof of (\ref{equation_criMinL}) for $D_l$, the proof of (\ref{equation_criMinS}) is analogous. So suppose that $\gamma$ is a proper minimal element of $M(F)$ and let $\beta$ be a long root in $V$. Then, by what already proved, there exists $\delta\in\Psi_\alpha$ and $\tau\in W_{\Delta\setminus\{\alpha\}}$ such that $0<(\nabla_{\alpha,\delta}^\tau,\gamma)<(\nabla_{\alpha,\delta}^\tau,\beta)$.

Now notice that $(\nabla_{\alpha,\delta}^\tau,\gamma)=(\nabla_{\alpha,\delta},\tau^{-1}\gamma)$ since $\tau$ stabilizes $\coomega_\alpha$. Further $\tau^{-1}\gamma\in R$ and we find that $(\coomega_\alpha,\tau^{-1}\gamma)$ and $(\coomega_\delta,\tau^{-1}\gamma)$ are integer. So $(\nabla_{\alpha,\delta}^\tau,\gamma)=(\frac{\coomega_\alpha}{m_\alpha},\tau^{-1}\gamma)-(\frac{\coomega_\delta}{m_\delta},\tau^{-1}\gamma)$ is an integral multiple of $1/\lcm(m_\alpha,m_\delta)$; further it is non negative as proved above.

Finally $(\nabla_{\alpha,\delta}^\tau,\beta)=(\nabla_{\alpha,\delta},\tau^{-1}\beta)\leq D_l(\delta)$ since $\tau$ stabilizes $V\cap\Phi_l$. This finishes the proof of (\ref{equation_criMinL}).
\end{proof}

The following lemma will be used for proving the subsequent necessary conditions for minimality. Let $\delta$, $\delta'$ be two different simple roots. Being $\Phi$ irreducible, the Dynkin diagram of $\Phi$ is a tree, hence there exists the minimal connected subset of $\Delta$ containing $\delta$ and $\delta'$; it is a segment of simple roots adjacent by pairs that we denote by $[\delta,\delta']$.

\begin{lemma}\label{lemma_chain}
If $\delta$, $\delta'$ are simple roots and $\tau\in W$ is such that $(\tau\coomega_\delta,\omega_{\delta'})<(\coomega_\delta,\omega_{\delta'})$ then $\coomega_\delta-\tau\coomega_\delta\geq\sum_{\epsilon\in[\delta,\delta']}\coepsilon$.
\end{lemma}
\begin{proof}
Let $\tau'\doteq s_{\alpha_r}s_{\alpha_{r-1}}\cdots s_{\alpha_1}$, with $\alpha_1,\alpha_2,\ldots,\alpha_r$ simpe roots, be a reduced expression for the minimal representative $\tau'$ of $\tau$ in $W/W_\delta$. Notice that defining $\lambda_i\doteq s_{\alpha_i}\cdots s_{\alpha_1}\coomega_\delta$, for $i=0,1,\ldots,r$, we have $(\lambda_i,\alpha_{i+1})>0$ for all $i=0,1,\ldots,r-1$ by (iv) of Theorem 4.3.1 in \cite{brentiBjorner}.

We will show by induction on $i$ that $\{\alpha_1,\alpha_2,\ldots,\alpha_i\}\subset\Delta$ is connected in the Dynkin diagram. The claim is clearly true for $i=1$ so suppose $i>1$. We have $(\lambda_i,\alpha_{i+1})>0$, hence we may consider the minimum $j\geq0$ such that $(\lambda_j,\alpha_{i+1})>0$. If $j=0$ then $\alpha_{i+1}=\delta=\alpha_1$ since $\lambda_0=\coomega_\delta$; in particular $\{\alpha_1,\alpha_2,\ldots,\alpha_{i+1}\}=\{\alpha_1,\alpha_2,\ldots,\alpha_i\}$ which is connected by induction.

So suppose $j>0$. Hence $(\lambda_{j-1},\alpha_{i+1})\leq0$. We have $\lambda_j=s_{\alpha_j}\lambda_{j-1}=\lambda_{j-1}-a\coalpha_j$, with $a\doteq(\lambda_{j-1},\alpha_j)>0$. Further $(\lambda_j,\alpha_{i+1})=(\lambda_{j-1},\alpha_{i+1})-a(\alpha_j,\coalpha_{i+1})>0$ so $(\alpha_j,\coalpha_{i+1})<0$ which implies that $\alpha_j$ and $\alpha_{i+1}$ are connected in the Dynkin diagram of $\Phi$. So $\{\alpha_1,\alpha_2,\ldots,\alpha_{i+1}\}$ is connected and our claim is proved.

Finally notice that $\delta'\in\{\alpha_1,\ldots,\alpha_r\}$ since otherwise
$$
(\tau\coomega_\delta,\omega_{\delta'})=(\tau'\coomega_\delta,\omega_{\delta'})=(\coomega_\delta,\tau'^{-1}\omega_{\delta'})=(\coomega_\delta,\omega_{\delta'}).
$$
So $\{\delta,\delta'\}\subseteq\{\alpha_1,\ldots,\alpha_r\}$ and this shows that $[\delta,\delta']\subseteq\{\alpha_1,\ldots,\alpha_r\}$ since this last set is connected. So $\coomega_\delta-\tau\coomega_\delta\geq\coalpha_1+\coalpha_2+\cdots+\coalpha_r\geq\sum_{\epsilon\in[\delta,\delta']}\coepsilon$.
\end{proof}

If $\Phi$ is not of type $\typeA$ then there exists a simple root, that we denote by $\nu$, such that $\theta$ is a positive multiple of $\omega_\nu$.

\begin{proposition}\label{proposition_testMinimal}
Suppose that $\Phi$ is not of type $\typeA$, let $\gamma$ be a $(\Delta\setminus\{\alpha\})$--dominant proper minimal element of $M(F)$ and write $\gamma=\sum_{\epsilon\in\Delta}c_\epsilon\omega_\epsilon$. Then the element $\delta\in\Psi_\alpha$ provided by Proposition \ref{proposition_criterionMinimal} for $\beta=\theta$ verifies
\begin{equation}\label{equation_testMin}
(\nabla_{\alpha,\delta},\gamma)+\frac{1}{m_\delta}\sum_{\epsilon\in[\delta,\nu]}c_\epsilon<D_{l}(\delta).
\end{equation}
\end{proposition}
\begin{proof} Let $\tau\in W_{\Delta\setminus\{\alpha\}}$ be the element provided by Proposition \ref{proposition_criterionMinimal} for $\beta=\theta$; so we have $0<(\nabla_{\alpha,\delta}^\tau,\gamma)<(\nabla_{\alpha,\delta}^\tau,\theta)=1-(\frac{\tau\coomega_\delta}{m_\delta},\theta)$. In particular we must have $(\tau\coomega_\delta,\theta)<m_\delta=(\coomega_\delta,\theta)$ and we may apply Lemma \ref{lemma_chain} with $\delta'=\nu$ (since $\theta=k\omega_\nu$ with $k\in\Q^+$) and conclude that $\coomega_\delta - \tau\coomega_\delta\geq\sum_{\epsilon\in[\delta,\nu]}\coepsilon$. Hence
$$
\begin{array}{rcl}
(\nabla_{\alpha,\delta}^\tau,\gamma) & = & (\frac{\coomega_\alpha}{m_\alpha},\gamma)-(\frac{\tau\coomega_\delta}{m_\delta},\gamma)\\
 & \geq & (\frac{\coomega_\alpha}{m_\alpha},\gamma)-(\frac{\coomega_\delta}{m_\delta},\gamma)+\sum_{\epsilon\in[\delta,\nu]}(\frac{\coepsilon}{m_\delta},\gamma)\\
 & = & (\nabla_{\alpha,\delta},\gamma)+\frac{1}{m_\delta}\sum_{\epsilon\in[\delta,\nu]}c_\epsilon
\end{array}
$$
and our claim follows since, as seen in the end of the proof of Proposition \ref{proposition_criterionMinimal}, $(\Delta_{\alpha,\delta}^\tau,\theta)\leq D_l(\delta)$.
\end{proof}

A similar result holds also for $D_s$ but we will not need it.

For proper elements that are invariant under the stabilizer of $F$ we have a very simple criterion of minimality.

\begin{proposition}\label{proposition_invariantMinimal}
An element $\gamma\in M(F)$ that is $W_{\Delta\setminus\{\alpha\}}$--invariant is proper minimal if and only if there exists $\delta_s,\delta_l\in\Psi_\alpha$ such that $0<(\nabla_{\alpha,\delta_s},\gamma)<D_s(\delta_s)$ and $0<(\nabla_{\alpha,\delta_l},\gamma)<D_l(\delta_l)$.
\end{proposition}
\begin{proof} For a $W_{\Delta\setminus\{\alpha\}}$--invariant element $\gamma$ we have $(\nabla_{\alpha,\delta}^\tau,\gamma)=(\nabla_{\alpha,\delta},\gamma)$ for any $\tau\in W_{\Delta\setminus\{\alpha\}}$. Hence the claim follows at once by Proposition \ref{proposition_criterionMinimal}.
\end{proof}

We are now ready to see the proof of Theorem \ref{introTheorem_properMinimalElements}. Although we have seen in this section some general conditions for proper minimal elements, we still need a case by case analysis. This combinatorial problem has some similarity with the classification of low triples in \cite{CM_projectiveNormality}. In Table \ref{table_rootData} we have some simple numerical data we will use in this analysis; we will tacitly use these data without mention the reference to the table.

The data in the table are: in the first column the irreducible type, in the second column the maximal roots, the third column gives the coefficient $m_\alpha$ for $\alpha$ a maximal root, the fourth the values of $\nabla_{\alpha,\delta}$ on $V_l(\alpha)$ and $V_s(\alpha)$ for any $\delta\neq\alpha$ maximal, the fifth, for the facets with autointersection, the value of $\nabla_{\alpha,\epsilon}$ on $V_l(\alpha)$ and $V_s(\alpha)$ where $\epsilon$ is the unique simple root adjacent to $\alpha$ and, finally, in the last column we have the proper minimal elements of the facet $F(\alpha)$ (these are computed in the proof). Notice that the fourth and fifth columns gives the value of $\nabla_{\alpha,\delta}$ for any $\delta\in\Psi_\alpha$; we have divided them in two columns for clarity. These values may be easily computed using Lemma \ref{lemma_facetOrbits}.

\begin{proof}[Proof of Theorem \ref{introTheorem_properMinimalElements}.] We assume in what follows that $\gamma=\sum_{\epsilon}c_\epsilon\omega_\epsilon$ is a $(\Delta\setminus\{\alpha\})$--dominant proper minimal element of $F\doteq F(\alpha)$, $\alpha$ a maximal root. So $c_\epsilon\geq0$ for all $\epsilon\in\Delta$, $\epsilon\neq\alpha$. It will turn out that any proper minimal element is $W_{\Delta\setminus\{\alpha\}}$--invariant, hence we compute all proper minimal elements and not only the dominant ones.

$\bullet$ Type $\typeA_\ell$. All simple roots are maximal. Let $\alpha\doteq\alpha_i$, $i=1,2,\ldots,\ell$; then $\Psi_\alpha$ is the set of simple roots adjacent to $\alpha_i$ in $\Delta$ and we have $\lcm(m_\alpha,m_\delta)=1$ for any $\delta\in\Psi_\alpha$. But notice that $D_l(\delta)=1$ too, so the necessary condition (\ref{equation_criMinL}) in Proposition \ref{proposition_criterionMinimal} cannot be fulfilled.

$\bullet$ Type $\typeB_\ell$, $\ell\geq3$. The two maximal roots are $\alpha_1$, $\alpha_\ell$. Let $\alpha\doteq\alpha_1$, then $\Psi_\alpha=\{\alpha_\ell\}$ and $m_1=1$, $m_\ell=2$. But $1/\lcm(m_1,m_\ell)=1/2=D_s(\alpha_\ell)$ and so the condition (\ref{equation_criMinS}) in Proposition \ref{proposition_criterionMinimal} is violated. We conclude that the face $F(\alpha_1)$ has no proper minimal element.

Now let $\alpha\doteq\alpha_\ell$. We have $\Psi_\alpha=\{\alpha_1,\alpha_{\ell-1}\}$ and $m_1=1$, $m_{\ell-1}=2$. Notice that $1/\lcm(m_\ell,m_{\ell-1})=1/2=D_{l}(\alpha_{\ell-1})$ hence the necessary condition (\ref{equation_criMinL}) of Proposition \ref{proposition_criterionMinimal} should be satisfied for $\delta=\alpha_1$. So we have $(\nabla_{\alpha_\ell,\alpha_1},\gamma)+c_1+c_2<1=D_{l}(\alpha_1)$, by Proposition \ref{proposition_testMinimal}, and so $c_1=c_2=0$. Moreover
$$
\nabla_{\alpha_\ell,\alpha_1}=\frac{\coomega_\ell}{2}-\coomega_1=\sum_{i=1}^{\ell-1}\frac{i-2}{2}\coalpha_{i}+\frac{\ell-2}{4}\coalpha_\ell
$$
hence
$$
(\nabla_{\alpha_\ell,\alpha_1},\gamma)=\sum_{i\geq3}^{\ell-1}\frac{i-2}{2}c_i+\frac{\ell-2}{4}c_\ell<1.
$$
Since $\gamma\in R$ we find that $c_\ell$ is even. Moreover, by a direct computation, $\nabla_{\alpha_\ell,\alpha_{\ell-1}}=\coomega_\ell/4$, hence $(\nabla_{\alpha,\alpha_{\ell-1}},\gamma)=c_\ell/4$ and, by Proposition \ref{proposition_criterionProper}, we have $0<c_\ell/4$; so $c_\ell$ is also positive.

Now this implies $\ell=3$ since for $\ell\geq4$ we have $(\ell-2)c_\ell/4\geq1$. Hence we find $c_\ell<4$ and so $c_\ell=2$. But now it is clear that $\gamma=2\omega_3$ is proper minimal since it is invariant by $W_{\Delta\setminus\{\alpha_3\}}$ and $\delta_s=\delta_l=\alpha_1$ fulfils the condition in Proposition \ref{proposition_invariantMinimal}.

In particular we conclude that the facet $F(\alpha_\ell)$ has no proper minimal element for $\ell\geq4$.

$\bullet$ Type $\typeC_\ell$. The unique maximal root is $\alpha_\ell$ and $\Psi_{\alpha_\ell}=\{\alpha_{\ell-1}\}$. Moreover $m_\ell=1$ and $m_{\ell-1}=2$. So, by condition (\ref{equation_criMinS}) of Proposition \ref{proposition_criterionMinimal}, there are no proper minimal element since $1/\lcm(m_\ell,m_{\ell-1})=1/2=D_s(\alpha_{\ell-1})$.

$\bullet$ Type $\typeD_\ell$. The maximal roots are $\alpha_1,\alpha_{\ell-1},\alpha_\ell$ and $m_\alpha=1$ for all such roots. Also these facets have no proper minimal element since $1/\lcm(m_\alpha,m_\delta)=1=D_{l}(\delta)$ for all maximal root $\delta\neq\alpha$ maximal and no facet has autointersection.

$\bullet$ Type $\typeE_6$. The maximal roots are $\alpha_1$ and $\alpha_6$; since they are symmetric we consider only $\alpha\doteq\alpha_1$. We have $\Psi_\alpha=\{\alpha_3,\alpha_6\}$. Since $m_1=m_6=1$ we find $1/\lcm(m_1,m_6)=1=D_{l}(\alpha_6)$; so the necessary condition (\ref{equation_criMinL}) of Proposition \ref{proposition_criterionMinimal} should be satisfied for $\delta=\alpha_3$ and $(\nabla_{\alpha_1,\alpha_3},\gamma)+(c_2+c_3+c_4)/2<1$ by Proposition \ref{proposition_testMinimal}. But $(\nabla_{\alpha_1,\alpha_3},\gamma)=c_1/2>0$ and we find $c_1=1$, $c_2=c_3=c_4=0$. Moreover, by Proposition \ref{proposition_criterionProper}, $0<(\nabla_{\alpha_1,\alpha_6},\gamma)=2(1-c_6)/3-c_5/3$ since
$$
\nabla_{\alpha_1,\alpha_6}=\coomega_1-\coomega_6=\frac{2}{3}(\coalpha_1-\coalpha_6)+\frac{1}{3}(\coalpha_3-\coalpha_5).
$$
This shows that $c_6=0$ and $c_5=0$ or $c_5=1$. We conclude either $\gamma=\omega_1$ or $\gamma=\omega_1+\omega_5$; but this is impossible since these are not elements of $R$. Hence there are no proper minimal elements for any facet for this type.

$\bullet$ Type $\typeE_7$. The maximal roots are $\alpha_2$ and $\alpha_7$. We consider first $F(\alpha_2)$. We have $\Psi_{\alpha_2}=\{\alpha_7\}$ and $m_2=2$, $m_7=1$. By Proposition \ref{proposition_testMinimal} $(\nabla_{\alpha_2,\alpha_7},\gamma)+c_1+c_3+c_4+c_5+c_6+c_7<1$, hence $c_1=c_3=c_4=c_5=c_6=c_7=0$. Further
$$
\nabla_{\alpha_2,\alpha_7}=\frac{1}{2}\coomega_2-\coomega_7=\frac{1}{4}(\coalpha_2-\coalpha_5-2\coalpha_6-3\coalpha_7)
$$
and so $c_2<4$. Moreover $0<(\nabla_{\alpha_2,\alpha_7},\gamma)=c_2/4$ and finally $\gamma=2\omega_2$ using the fact that $\gamma\in R$. This element is minimal by Proposition \ref{proposition_invariantMinimal} with $\delta_l=\alpha_7$.

Now consider $F(\alpha_7)$. We have $\Psi_{\alpha_7}=\{\alpha_2,\alpha_6\}$, $m_7=1$, $m_2=2$, $m_6=2$. We will prove that the condition (\ref{equation_testMin}) of Proposition \ref{proposition_testMinimal} does not hold for $\delta=\alpha_2$ nor $\delta=\alpha_6$. We begin considering the case $\delta=\alpha_2$.

Indeed in this case we have
$$
\nabla_{\alpha_7,\alpha_2}=\coomega_7-\frac{1}{2}\coomega_2=\frac{3}{4}\coalpha_7+\frac{1}{4}\coalpha_5-\frac{1}{4}\coalpha_2,
$$
hence, by Proposition \ref{proposition_criterionMinimal} we should have $(\nabla_{\alpha_7,\alpha_2},\gamma)+(c_1+c_2+c_3+c_4)/2=c_1/2+c_2/4+c_3/2+c_4/2+c_5/4+3c_7/4<1$. But $0<(\nabla_{\alpha_7,\alpha_6},\gamma)=(\coomega_7-\coomega_6/2,\gamma)=(\coalpha_7,\gamma)/2=c_7/2$. Hence $c_7=1$, $c_1=c_2=c_3=c_4=c_5=0$ and $\gamma=c_6\omega_6+\omega_7$. However no such $\gamma\in R$ and so we have proved that the condition (\ref{equation_criMinL}) in Proposition \ref{proposition_criterionMinimal} does not hold for $\delta=\alpha_2$.

Now suppose $\delta=\alpha_6$ in Proposition \ref{proposition_criterionMinimal} and let $\tau$ be an element in $W_{\Delta\setminus\{\alpha_7\}}$ for which the condition of the proposition is fulfilled. As we have just seen $(\nabla_{\alpha_7,\alpha_6},\gamma)=c_7/2$, hence
$$
0<\frac{c_7}{2}=(\nabla_{\alpha_7,\alpha_6},\gamma)\leq(\nabla_{\alpha_7,\alpha_6}^\tau,\gamma)<(\nabla_{\alpha_7,\alpha_6}^\tau,\theta)=1-(\frac{\coomega_6}{2},\tau^{-1}\theta)
$$
using Lemma \ref{lemma_dominantInequality} for the second inequality from the left. Hence $c_7=1$ and $(\coomega_6,\tau^{-1}\theta)=0$.

Notice that $\tau^{-1}\theta\in V(\alpha_7)$ since $\tau\in W_{\Delta\setminus\{\alpha_7\}}$; so $\alpha_7$ is in the support of the root $\tau^{-1}\theta$ while $\alpha_6$ is not in this support since $(\coomega_6,\tau^{-1}\theta)=0$. But the support of a root is a connected subset of $\Delta$, so we find $\tau^{-1}\theta=\alpha_7=-\coomega_6+2\coomega_7$.

Since $\theta=\coomega_1$ we have $\tau\coomega_6=-\coomega_1+2\coomega_7=\coomega_6-(2\coalpha_1+2\coalpha_2+3\coalpha_3+4\coalpha_4+3\coalpha_5+2\coalpha_6)$ (see the table for $\typeE_7$ in \cite{bourbaki}, for example). We find
$$
\begin{array}{rcl}
(\nabla_{\alpha_7,\alpha_6}^\tau,\gamma) & = & (\coomega_7,\gamma)-(\frac{\tau\coomega_6}{2},\gamma)\\
 & = & (\coomega_7,\gamma)-(\frac{\coomega_6}{2},\gamma)+(\frac{\coomega_6-\tau\coomega_6}{2},\gamma)\\
 & = & (\nabla_{\alpha_7,\alpha_6},\gamma)+\frac{1}{2}(2\coalpha_1+2\coalpha_2+3\coalpha_3+4\coalpha_4+3\coalpha_5+2\coalpha_6,\gamma)\\
 & = & \frac{c_7}{2}+c_1+c_2+\frac{3}{2}c_3+2c_4+\frac{3}{2}c_5+c_6.
\end{array}
$$
Now, using $c_7=1$ and $(\nabla_{\alpha_7,\alpha_6}^\tau,\theta)=1$ proved above, the condition of Proposition \ref{proposition_criterionMinimal} becomes
$$
\frac{1}{2}+c_1+c_2+\frac{3}{2}c_3+2c_4+\frac{3}{2}c_5+c_6<1.
$$
Hence $c_1=c_2=c_3=c_4=c_5=c_6=0$. So $\gamma=\coomega_7$, but $\coomega_7\not\in R$ and this shows that we cannot have $\delta=\alpha_6$ in Proposition \ref{proposition_criterionMinimal} either.

We conclude that the facet $F(\alpha_7)$ has no proper minimal element.

$\bullet$ Type $\typeE_8$. There are two maximal roots: $\alpha_1$ and $\alpha_2$. We begin with $F(\alpha_1)$ showing that this facet has no proper minimal element.

We have $\Phi_{\alpha_1}=\{\alpha_2,\alpha_3\}$. Since $1/\lcm(m_1,m_3)=1/\lcm(2,4)=1/4=D_l(\alpha_3)$, we cannot have $\delta=\alpha_3$ by condition (\ref{equation_criMinL}) in Proposition \ref{proposition_criterionMinimal}. For $\delta=\alpha_2$, by Proposition \ref{proposition_testMinimal}, $(\nabla_{\alpha_1,\alpha_2},\gamma)+(c_2+c_4+c_5+c_6+c_7+c_8)/3<1/3=D_l(\alpha_2)$ since $m_2=3$. So $c_2=c_4=c_5=c_6=c_7=c_8=0$, moreover
$$
\nabla_{\alpha_1,\alpha_2}=\frac{1}{2}\coomega_1-\frac{1}{3}\coomega_2=\frac{1}{3}\coalpha_1-\frac{1}{6}\coalpha_2+\frac{1}{6}\coalpha_3
$$
and so $c_1/3+c_3/6<1/3$. But $0<(\nabla_{\alpha_1,\alpha_2},\gamma)=c_1/4$ which implies $c_1>0$ that is incompatible with the previous inequality. This proves our claim about $F(\alpha_1)$.

Now we consider $F(\alpha_2)$; here $\Psi_{\alpha_2}=\{\alpha_1\}$. The condition in Proposition \ref{proposition_testMinimal} is
$$
(\nabla_{\alpha_2,\alpha_1},\gamma)+(c_1+c_3+c_4+c_5+c_6+c_7+c_8)/2<1/2
$$
hence $c_1=c_3=c_4=c_5=c_6=c_7=c_8=0$ and we have $\gamma=c_2\omega_2$. We find $0<(\nabla_{\alpha_2,\alpha_1},\gamma)=c_2/6<1/2$ by which $c_2=1$ or $c_2=2$. Since $\omega_2\in R$ and it is invariant by $W_{\Delta\setminus\{\alpha_1\}}$, $\omega_2$ and $2\omega_2$ are minimal element by Proposition \ref{proposition_invariantMinimal} with $\delta_l=\alpha_1$.

$\bullet$ Type $\typeF_4$. The unique maximal root is $\alpha_4$ with $\Psi_{\alpha_4}=\{\alpha_3\}$ but $1/\lcm(m_4,m_3)=1/\lcm(2,4)=1/4=D_s(\alpha_3)$, and hence condition (\ref{equation_criMinS}) in Proposition \ref{proposition_criterionMinimal} cannot be fulfilled. So $M(F(\alpha_4))$ has no proper minimal element.

$\bullet$ Type $\typeG_2$. Only $\alpha_1$ is maximal and $\delta=\alpha_2$ in Proposition \ref{proposition_testMinimal}. We have $m_1=3$, $m_2=2$ and so the condition is $0<(\nabla_{\alpha_1,\alpha_2},\gamma)+c_2/2<1/2=D_l(\alpha_2)$. We find $c_2=0$, hence $0<(\nabla_{\alpha_1,\alpha_2},\gamma)=c_1/6<1/2$. We conclude that the proper minimal elements are $\omega_1$, $2\omega_1$ by Proposition \ref{proposition_invariantMinimal}.
\end{proof}

If $F$ is a face of the root polytope, let us define $\mi(F)$ as the set of non zero $\leq_F$--minimal elements of $M(F)$ with respect to the order defined in Section \ref{section_preliminary} and define $\prmi(F)\subseteq\mi(F)$ as the set of proper minimal elements of $M(F)$.

\begin{corollary}\label{corollary_minimalOutside}
For any face $F$ of the root polytope $\mi(F)\cap Z(V(F))=\varnothing$.
\end{corollary}
\begin{proof}
First we show that $\prmi(F)\cap Z(V(F))=\varnothing$.

The face $F$ is a facet of the root polytope for some irreducible subsystem of $\Phi$ by (1) of Corollary 4.5 in \cite{celliniMarietti} and Lemma \ref{lemma_subfaceSubsystem}. Further, by the Weyl group action, we may assume also that $F$ is a coordinate facet, so $F\doteq F(\alpha)$ for a maximal root $\alpha\in\Delta$. So it suffices to prove our claim for the proper minimal elements in the Table \ref{table_rootData}. For each such an element $\gamma$ we check that $m_\alpha$ does not divide $(\coomega_\alpha,\gamma)$; this is clearly sufficient to conclude that $\gamma\not\in Z(V(\alpha))$.

For $\typeB_3$, $\alpha\doteq\alpha_3$ and $(\coomega_3,2\omega_3)=3$ and our claim is true since $m_3=2$. For $\typeE_7$, $\alpha\doteq\alpha_2$ and $(\coomega_2,2\omega_2)=7$ while $m_2=2$. For $\typeE_8$, $\alpha\doteq\alpha_2$ and $(\coomega_2,\omega_2)=8$ while $m_2=3$, hence $\omega_2$ and $2\omega_2$ are not elements of $Z(V(\alpha))$. Finally for $\typeG_2$, $\alpha\doteq\alpha_1$, we have $(\coomega_1,\omega_1)=2$ and we conclude that $\omega_1$ and $2\omega_1$ are not in $Z(V(\alpha))$.

Now notice that $\mi(F)=\cup\prmi(F')$ where $F'$ runs in the set of subfaces of $F$. So if $\gamma\in\mi(F)\cap Z(V(F))$, there exists $F'$ subface of $F$ such that $\gamma\in\prmi(F')\cap Z(V(F))$. But $\prmi(F')\subset\langle F'\rangle_\R$ and so $\gamma\in Z(V(F'))$ by Lemma \ref{lemma_subfaceLattice}. This finishes the proof since $\prmi(F')\cap Z(V(F'))=\varnothing$ as proved above.
\end{proof}

We are now ready to prove the Corollaries \ref{introCorollary_faceNormality} and \ref{introCorollary_faceIntegrality} in the Introduction.

\begin{proof}[Proof of Corollary \ref{introCorollary_faceNormality}.] Let $F$ be a face of the root polytope; we have to show that $C(V(F))\cap Z(V(F))=N(V(F))$.

It is clear that the set in the right hand side is a subset of the one in the left hand side so let $\gamma\in C(V(F))\cap Z(V(F))\subset C(V(F))\cap R=M(F)$. 

There exists $\gamma_0\in N(V(F))$ and $\gamma_1\in\mi(F)\cup\{0\}$ such that $\gamma=\gamma_0+\gamma_1$. So $\gamma_1=\gamma-\gamma_0\in Z(V(F))$ and we find $\gamma_1=0$ since $\mi(F)\cap Z(V(F))=\varnothing$ by Corollary \ref{corollary_minimalOutside}. Hence $\gamma=\gamma_0\in N(V(F))$.
\end{proof}

\begin{proof}[Proof of Corollary \ref{introCorollary_faceIntegrality}.] Let $F\doteq F(\alpha)$, $\alpha$ a maximal root, denote by $\mathcal{R}$ a system of representatives for the quotient $R/Z(V(F))$ and suppose $0$ represents the class of $Z(V(F))$. Then if we denote by $M_\gamma$ the intersection $(\gamma+Z(V(F)))\cap C(V(F))$ we have $M(F)=R\cap C(V(F))=\cup M_\gamma$ where $\gamma$ runs in $\mathcal{R}$. Now notice that $M_\gamma$ is a non--void set since $C(V(F))$ is a cone of maximal dimension (i.e. it spans $E$ as a vector space).

Further we know that $M_0=Z(V(F))\cap C(V(F))=N(V(F))$ by Corollary \ref{introCorollary_faceNormality}. Since, by definition, the facet $F$ is integrally closed in $R$ if and only if $R\cap C(V(F))=N(V(F))$, we find at once that this is the case if and only if $R=Z(V(F))$. Now by Lemma \ref{lemma_faceLattice} we have $Z(V(F))=\langle\Delta\setminus\{\alpha\},\beta_{\{\alpha\}}\rangle_\Z=\langle\Delta\setminus\{\alpha\},m_\alpha\alpha\rangle_\Z$. So $R=Z(V(F))$ if and only if $m_\alpha=1$ and this is our claim.
\end{proof}

\section{Application to the length map}\label{section_applicationLength}

In this section we prove the main Theorem \ref{introTheorem_formula} of our paper giving the formula for the length map. We begin by checking a particular case of the formula for the minimal elements of a facet.

\begin{lemma}\label{lemma_minimalFormula} If $F$ is a facet of the root polytope and $\gamma\in\mi(F)$ then $|\gamma|\leq\lceil(\lambda_F,\gamma)\rceil$ where $\lambda_F$ is the vector defining $F$.
\end{lemma}
\begin{proof} Since $\gamma$ is a minimal element of $M(F)$ there exists a subface $F'$ of $F$ such that $\gamma$ is a proper minimal element of $M(F')$. Moreover by Lemma \ref{lemma_subface}, $F'$ is a facet of the root polytope of $\Phi'$, where $E'\doteq\langle F'\rangle_\R$, $\Phi'\doteq\Phi\cap E'$, and it is clearly defined by $\lambda_{F|E'}$. Hence we may assume that $\gamma\in\prmi(F)$.

If $\tau\in W$ then $\tau\cdot\gamma$ is a minimal element of $M(\tau F)$; further $\lambda_{\tau F}=\tau\lambda_F$ and $|\tau\gamma|=|\gamma|$. So we may also assume that $F\doteq F(\alpha)$ is a coordinate face for a maximal root $\alpha\in\Delta$; in particular $F$ is defined by $\lambda_F=\coomega_\alpha/m_\alpha$. Hence we complete the proof by checking the inequality in the claim writing the proper minimal elements in the Table \ref{table_rootData} as sum of roots.

For $\typeB_3$ and $\alpha\doteq\alpha_3$,
$$
\begin{array}{rcl}
\gamma & = & 2\omega_3\\
 & = & \alpha_1+2\alpha_2+3\alpha_3\\
 & = & (\alpha_1+2\alpha_2+2\alpha_3)+\alpha_3
\end{array}
$$
so $|\gamma|\leq 2 = \lceil(\coomega_3,\gamma)/2\rceil$.

For $\typeE_7$ and $\alpha\doteq\alpha_2$,
$$
\begin{array}{rcl}
\gamma & = & 2\omega_2\\
 & = & 4\alpha_1+7\alpha_2+8\alpha_3+12\alpha_4+9\alpha_5+6\alpha_6+3\alpha_7\\
 & = & \theta+(\theta-\alpha_1-\alpha_3)+\theta_{\typeE_6}+(\alpha_2+\alpha_3+\alpha_4+\alpha_5+\alpha_6+\alpha_7)
\end{array}
$$
so $|\gamma|\leq 4=\lceil(\coomega_2,\gamma)/2\rceil$.

For $\typeE_8$ and $\alpha\doteq\alpha_2$, $\gamma=\omega_2=\theta+(\theta-\alpha_8)+\theta_{\typeE_6}$. So $|\gamma|\leq 3=\lceil\coomega_2(\gamma)/3\rceil$ and also $|2\gamma|\leq 6=\lceil(\coomega_2,2\gamma)/3\rceil$.

For $\typeG_2$ and $\alpha\doteq\alpha_1$, $\gamma=\omega_1=2\alpha_1+\alpha_2\in\Phi$. So $|\gamma|=1=\lceil(\coomega_1,\gamma)/3\rceil$ and $|2\gamma|\leq 2=\lceil(\coomega_1,2\gamma)/3\rceil$.
\end{proof}

We are now in a position to prove our formula for the length map.
\begin{proof}[Proof of Theorem \ref{introTheorem_formula}.] Let $\gamma\in R$ and $r\doteq |\gamma|$. Then there exist $\beta_1,\beta_2,\ldots,\beta_r\in\Phi$ such that $\gamma=\beta_1+\beta_2+\cdots+\beta_r$. Now if $\lambda_F$ defines the facet $F$ we have $(\lambda_F,\gamma)=(\lambda_F,\beta_1)+(\lambda_F,\beta_2)+\cdots+(\lambda_F,\beta_r)\leq r$ since $(\lambda_F,u)\leq1$ for all $u\in\RP_\Phi$. But being $r$ a non--negative integer we have also $\lceil(\lambda_F,\gamma)\rceil\leq r$.

In order to prove the reverse inequality notice that the $\Q$--cones over the facets partition $\langle\Delta\rangle_\Q$; hence if $\gamma\in R$ there exists a facet $F$ such that $\gamma\in C(V(F))$. Then $(\lambda_F,\gamma)\geq(\lambda_{F'},\gamma)$ for all facet $F'$, so it suffices to show that $|\gamma|\leq\lceil(\lambda_F,\gamma)\rceil$.

The monoid $C(V(F))\cap R=M(F)$ is the union of the monoids $\gamma'+N(V(F))$ with $\gamma'\in\mi(F)\cup\{0\}$. Hence there exists $\beta_1,\beta_2,\ldots,\beta_t\in V(F)\subset\Phi$ and $\gamma'\in\mi(F)\cup\{0\}$ such that $\gamma=\beta_1+\beta_2+\cdots+\beta_t+\gamma'$. So, using Lemma \ref{lemma_minimalFormula}, $|\gamma|\leq t+|\gamma'|\leq t+\lceil(\lambda_F,\gamma')\rceil=\lceil t+(\lambda_F,\gamma')\rceil=\lceil(\lambda_F,\gamma)\rceil$.
\end{proof}

As an application of our results we see an explicit formula for the (positive) length for elements of $R^+$ for type $\typeA_\ell$. Similar but more complex formulas may be derived for the other irreducible root systems. Since any partition in positive roots is a partition in roots, we have always $|\gamma|\leq|\gamma|_+$ for all $\gamma\in R^+$; in what follows we will use this many times without explicit mention.

Assuming that $\Phi$ is of type $\typeA_\ell$, we want to show that the partition of minimal size for an element $\gamma\in R^+$ is given by what we call a \emph{horizontal tiling}. We illustrate this with a graphical example for $\typeA_6$ and $\gamma=2\alpha_1+3\alpha_2+3\alpha_3+4\alpha_5+\alpha_6$. Consider the map $a:i\longmapsto a_i$ with $a_i=(\coomega_i,\gamma)$ that we draw as in the Figure \ref{figure_horizontalPartition}.
\begin{figure}[h]
\centering
\begin{pgfpicture}{0cm}{0cm}{4cm}{2.5cm}
\begin{pgftranslate}{\pgfpoint{0cm}{0.5cm}}
\pgfsetxvec{\pgfpoint{0.5cm}{0cm}}
\pgfsetyvec{\pgfpoint{0cm}{0.5cm}}
\begin{pgfscope}
\color{lightgray}
\pgfline{\pgfxy(2,0)}{\pgfxy(2,2)}
\pgfline{\pgfxy(3,0)}{\pgfxy(3,3)}
\pgfline{\pgfxy(6,0)}{\pgfxy(6,1)}
\end{pgfscope}
\pgfline{\pgfxy(0,0)}{\pgfxy(8,0)}
\pgfsetlinewidth{0.8pt}
\pgfline{\pgfxy(1,0)}{\pgfxy(4,0)}
\pgfline{\pgfxy(5,0)}{\pgfxy(7,0)}
\pgfline{\pgfxy(1,0)}{\pgfxy(1,2)}
\pgfline{\pgfxy(2,2)}{\pgfxy(2,3)}
\pgfline{\pgfxy(4,0)}{\pgfxy(4,3)}
\pgfline{\pgfxy(5,0)}{\pgfxy(5,4)}
\pgfline{\pgfxy(6,1)}{\pgfxy(6,4)}
\pgfline{\pgfxy(7,0)}{\pgfxy(7,1)}
\pgfline{\pgfxy(1,1)}{\pgfxy(4,1)}
\pgfline{\pgfxy(1,2)}{\pgfxy(4,2)}
\pgfline{\pgfxy(2,3)}{\pgfxy(4,3)}
\pgfline{\pgfxy(5,1)}{\pgfxy(7,1)}
\pgfline{\pgfxy(5,2)}{\pgfxy(6,2)}
\pgfline{\pgfxy(5,3)}{\pgfxy(6,3)}
\pgfline{\pgfxy(5,4)}{\pgfxy(6,4)}
\pgfputat{\pgfxy(1.5,-0.5)}{\pgfbox[center,center]{$1$}}
\pgfputat{\pgfxy(2.5,-0.5)}{\pgfbox[center,center]{$2$}}
\pgfputat{\pgfxy(3.5,-0.5)}{\pgfbox[center,center]{$3$}}
\pgfputat{\pgfxy(4.5,-0.5)}{\pgfbox[center,center]{$4$}}
\pgfputat{\pgfxy(5.5,-0.5)}{\pgfbox[center,center]{$5$}}
\pgfputat{\pgfxy(6.5,-0.5)}{\pgfbox[center,center]{$6$}}
\end{pgftranslate}
\end{pgfpicture}
\caption{A horizontal partition}
\label{figure_horizontalPartition}
\end{figure}
We define a partition of $\gamma$ by grouping together as many boxes as possible in horizontal lines. This is what we call the horizontal tiling for $\gamma$; it gives a partition of $\gamma$ in terms of positive roots that we call the \emph{horizontal partition}. It is not hard to show that the number $h(\gamma)$ of roots in the horizontal partition for $\gamma=\sum_{i=1}^\ell a_i\alpha_i\in R^+$ is given by the formula
$$
h(\gamma)=\sum_{i=1}^{\ell}\max(a_i-a_{i-1},0)
$$
where we set $a_0=0$. In the following lemma we show that both the length and the positive length are given by the map $h$; so in particular they coincide for type $\typeA_\ell$.

\begin{proposition}\label{proposition_lengthMapA}
If $\Phi$ is of type $\typeA_\ell$ then, for all $\gamma\in R^+$, $|\gamma|=|\gamma|_+=h(\gamma)$.
\end{proposition}
\begin{proof}
For $1\leq h\leq k\leq\ell$, we denote by $\alpha_{h,k}$ the root $\alpha_h+\alpha_{h+1}+\cdots+\alpha_k$. First of all notice that $|\gamma|_+\leq h(\gamma)$ since $h(\gamma)$ is the size of a partition in positive root (the horizontal one). Now we claim that $h(\gamma+\alpha_{h,k})\leq h(\gamma)+1$.

Indeed let $\gamma=\sum_{i=1}^\ell a_i\alpha_i$, $\gamma'\doteq\gamma+\alpha_{h,k}=\sum_{i=1}^\ell a'_i\alpha_i$. Notice that we have
$$
a'_i=\left\{
\begin{array}{ll}
a_i & \textrm{if }1\leq i<h\\
a_i+1 & \textrm{if }h\leq i\leq k\\
a_i & \textrm{if }k<i\leq\ell.
\end{array}
\right.
$$
Hence $h(\gamma')=h(\gamma)+(a_h-a_{h-1}+1)^0-(a_h-a_{h-1})^0+(a_{k+1}-a_k-1)^0-(a_{k+1}-a_k)^0$, where we set $(a)^0\doteq\max(a,0)$ for short. It is clear that $(a+1)^0-(a)^0\leq1$ and $(a-1)^0-(a)^0\leq0$; so we have $h(\gamma')\leq h(\gamma)+1$ as claimed.

Now let $\gamma=\beta_1+\beta_2+\cdots+\beta_r$ be a partition of minimal size of $\gamma$ in positive roots. Since any positive root is $\alpha_{h,k}$ for some $h,k$ as above, we find $h(\gamma)\leq r=|\gamma|_+$ by what proved. So $|\gamma|_+=h(\gamma)$; the horizontal partition is a partition of minimal size.

The final step is to show that $|\gamma|=|\gamma|_+$. We use the formula in Theorem \ref{introTheorem_formula}; for type $\typeA_\ell$ the facets are defined by the orbits of the fundamental weights under the Weyl group action. In turns they are in bijection with the set of sequences, called \emph{rows}, $R=1\leq r_1<r_2<\cdots r_j\leq\ell+1$ of increasing integers, with $j=1,2,\ldots,\ell$: the row $R$ corresponds to the weight $\lambda(R)\doteq\sum_{h=1}^j\omega_{r_h}-\omega_{r_h-1}$ (where we set $\omega_0=\omega_{\ell+1}=0$).

But for any such $R$ we have $(\lambda(R),\gamma) = \sum_{h=1}^j(\omega_{r_h}-\omega_{r_h-1},\gamma) = \sum_{h=1}^j(a_{r_h}-a_{r_h-1}) \leq \sum_{i=1}^l(a_i-a_{i-1})^0 = |\gamma|_+$. This proves that $|\gamma|=|\gamma|_+$.
\end{proof}

The next step is to prove that the two length maps are equal also for type $\typeC_\ell$.

\begin{proposition}\label{proposition_lengthMapC}
If $\Phi$ is of type $\typeC_\ell$, then for all $\gamma\in R^+$, $|\gamma|=|\gamma|_+$.
\end{proposition}
\begin{proof}
Let $F$ be a facet of $\RP_\Phi$ such that $\gamma\in M(F)$. In \cite{celliniMarietti2} a triangulation of the facets of $\RP_\Phi$ for type $\typeC_\ell$ is defined. This triangulation has the following two properties:
\begin{itemize}
\item[(i)] $C(V(F))\cap C(\Phi^+)$ is the union of the non--negative rational cones generated by certain simplexes of the triangulation;
\item[(ii)] each simplex of the triangulation of $F$ is a basis for the lattice $R$.
\end{itemize}
Hence $\gamma$ is a non--negative integral linear combination of a simplex $T\subset\Phi^+$ of the triangulation; let us say $\gamma=\sum_{\beta\in T}a_\beta\beta$, $a_\beta\in\N$. So we find $|\gamma|_+\leq\sum_{\beta\in T}a_\beta=(\lambda_F,\gamma)=|\gamma|$ and this finishes the proof.
\end{proof}

Now we prove that the positive length map is different from the length for all types but $\typeA_\ell$ and $\typeC_\ell$. 

\begin{proposition}\label{proposition_differentTypes}
If $\Phi$ is not of type $\typeA_\ell$ nor $\typeC_\ell$ then there exists $\gamma\in\R^+$ such that $|\gamma|<|\gamma|_+$.
\end{proposition}
\begin{proof} For all the above types we find $\alpha,\beta\in\Phi^+$ such that $\alpha-\beta\in R^+$, $|\alpha-\beta|_+=3$; this suffices to prove our claim since $|\alpha-\beta|\leq 2$. The table below reports such roots $\alpha$, $\beta$.
\begin{table}[h]
\centering
\begin{tabular}{l|c|c}
type & $\alpha$ & $\beta$\\
\hline
$\typeB_\ell$, $\ell\geq3$ & $\alpha_{\ell-2}+\alpha_{\ell-1}+2\alpha_\ell$ & $\alpha_{\ell-1}$\\
$\typeD_\ell$ & $\alpha_{\ell-3}+\alpha_{\ell-2}+\alpha_{\ell-1}+\alpha_\ell$ & $\alpha_{\ell-2}$\\
$\typeE$ & $\alpha_2+\alpha_3+\alpha_4+\alpha_5$ & $\alpha_4$\\
$\typeF_4$ & $\alpha_1+\alpha_2+2\alpha_3$ & $\alpha_2$\\
$\typeG_2$ & $3\alpha_1+\alpha_2$ & $\alpha_2$\\
\end{tabular}
\end{table}
\end{proof}

Finally the Corollary \ref{introCorollary_lengths} of the Introduction follows by Propositions \ref{proposition_lengthMapA}, \ref{proposition_lengthMapC} and \ref{proposition_differentTypes}.

\newpage

\begin{table}[h]
\centering
\begin{tabular}{l|c|c|c|c|c}
\multirow{2}{*}{type} &\multirow{2}{*}{maximal roots $\alpha$} & \multirow{2}{*}{$m_\alpha$} & values of $\nabla_{\alpha,\delta}$ on $V(\alpha)$ & values of $\nabla_{\alpha,\epsilon}$ on $V(\alpha)$ & minimal\\
 & & & $\delta\neq\alpha$ max. & $\epsilon$ non max. adj. to $\alpha$ & elements\\
\hline
\multirow{2}{*}{$\typeA_\ell$} & \multirow{2}{*}{$\alpha_1,\alpha_2,\ldots,\alpha_\ell$} & \multirow{2}{*}{$1$} & \multirow{2}{*}{$0$,$1$} & \multirow{2}{*}{---} & \multirow{2}{*}{---}\\
 & & & & \\
\hline
\multirow{4}{*}{$\typeB_\ell$, $\ell\geq3$} & \multirow{2}{*}{$\alpha_1$} & \multirow{2}{*}{$1$} & $0$, $1$ on $V_l(\alpha)$ & \multirow{2}{*}{---} & \multirow{2}{*}{---}\\
 & & & $1/2$ on $V_s(\alpha)$ & & \\
\cline{2-6}
 & \multirow{2}{*}{$\alpha_\ell$} & \multirow{2}{*}{$2$} & \multirow{2}{*}{$0$, $1$} & ($\epsilon=\alpha_{\ell-1},\, m_\epsilon = 2$) & \multirow{2}{*}{$2\omega_3$ for $\ell=3$}\\
 & & & & $0$, $1/2$\\
\hline
\multirow{3}{*}{$\typeC_\ell$, $\ell\geq2$} & \multirow{3}{*}{$\alpha_\ell$} & \multirow{3}{*}{$1$} & \multirow{3}{*}{---} & ($\epsilon=\alpha_{\ell-1},\, m_\epsilon = 2$) & \multirow{3}{*}{---}\\
 & & & & $0$, $1$ on $V_l(\alpha)$\\
 & & & & $0$, $1/2$ on $V_s(\alpha)$\\
\hline
\multirow{2}{*}{$\typeD_\ell$} & \multirow{2}{*}{$\alpha_1$, $\alpha_{\ell-1}$, $\alpha_\ell$} & \multirow{2}{*}{$1$} & \multirow{2}{*}{$0$, $1$} & \multirow{2}{*}{---} & \multirow{2}{*}{---}\\
 & & & & \\
\hline
\multirow{4}{*}{$\typeE_6$} & \multirow{2}{*}{$\alpha_1$} & \multirow{2}{*}{$1$} & \multirow{2}{*}{$0$, $1$} & ($\epsilon=\alpha_3$, $m_\epsilon=2$) & \multirow{2}{*}{---}\\
 & & & & $0$, $1/2$, $1$\\
\cline{2-6}
 & \multirow{2}{*}{$\alpha_6$} & \multirow{2}{*}{$1$} & \multirow{2}{*}{$0$, $1$} & ($\epsilon=\alpha_5$, $m_\epsilon=2$) & \multirow{2}{*}{---}\\
 & & & & $0$, $1/2$, $1$\\
\hline
\multirow{4}{*}{$\typeE_7$} & \multirow{2}{*}{$\alpha_2$} & \multirow{2}{*}{$2$} & \multirow{2}{*}{$0$, $1$} & \multirow{2}{*}{---} & \multirow{2}{*}{$2\omega_2$}\\
 & & & &\\
\cline{2-6}
 & \multirow{2}{*}{$\alpha_7$} & \multirow{2}{*}{$1$} & \multirow{2}{*}{$0$, $1/2$, $1$} & ($\epsilon=\alpha_6$, $m_\epsilon=2$) & 	\multirow{2}{*}{---}\\
 & & & & $0$, $1/2$, $1$\\
\hline
\multirow{4}{*}{$\typeE_8$} & \multirow{2}{*}{$\alpha_1$} & \multirow{2}{*}{$2$} & \multirow{2}{*}{$0$, $1/3$} & ($\epsilon=\alpha_3$, $m_\epsilon=4$) & \multirow{2}{*}{---}\\
 & & & & $0$, $1/4$\\
\cline{2-6}
 & \multirow{2}{*}{$\alpha_2$} & \multirow{2}{*}{$3$} & \multirow{2}{*}{$0$, $1/2$} & \multirow{2}{*}{---} & \multirow{2}{*}{$\omega_2,\,2\omega_2$}\\
 & & & & \\
\hline
\multirow{3}{*}{$\typeF_4$} & \multirow{3}{*}{$\alpha_4$} & \multirow{3}{*}{$2$} & \multirow{3}{*}{---} & ($\epsilon=\alpha_3$, $m_\epsilon=4$) & \multirow{3}{*}{---}\\
 & & & & $0$, $1/2$ on $V_l(\alpha)$\\
 & & & & $1/4$ on $V_s(\alpha)$\\
\hline
\multirow{2}{*}{$\typeG_2$} & \multirow{2}{*}{$\alpha_1$} & \multirow{2}{*}{$3$} & \multirow{2}{*}{---} & ($\epsilon=\alpha_2$, $m_\epsilon=2$) & \multirow{2}{*}{$\omega_1,\,2\omega_1$}\\
 & & & & $0$, $1/2$\\
\end{tabular}
\caption{Maximal root data}
\label{table_rootData}
\end{table}

{\footnotesize
{\sc Dipartimento di Matematica e Fisica, Universita' del Salento, Via per Arnesano, 73047 Monteroni di Lecce (LE), Italy}

{\it E-mail address}: {\tt rocco.chirivi@unisalento.it}
}

\end{document}